\documentclass[a4paper,12pt,reqno]
{amsart}
\pdfoutput=1
\usepackage{a4}
\usepackage{amsthm}
\usepackage{amssymb}
\usepackage{hyperref}

\newtheorem{Definition}{Definition}
\newtheorem{Proposition}{Proposition}
\newtheorem{Lemma}{Lemma}

\newtheorem{Remark}{Remark}

\begin{document}
	\title
	[Generalized Heisenberg-Virasoro algebra]
	{Generalized Heisenberg-Virasoro algebra and matrix models from quantum algebra}

	\author{Fridolin Melong$^{\dagger}$  \and  Raimar Wulkenhaar$^{\ddagger}$}
	
	\address{ Department of Mathematics, Faculty of Sciences,\\ University of Yaounde 1,
		P.O.Box 812, Yaounde, Cameroon \\ and International Chair in Mathematical Physics
		and Applications
		(ICMPA-UNESCO Chair) \\
	 University of Abomey-Calavi,
 072 B.P. 50 Cotonou, Republic of Benin
		{\itshape e-mail:} \normalfont  
		\texttt{fridomelong@gmail.com}}
	
	\address{Mathematisches Institut der
		Westfälischen Wilhelms-Universit\"at, \newline
		Einsteinstr.\ 62, 48149 M\"unster, Germany, 
		{\itshape e-mail:} \normalfont
		\texttt{raimar@math.uni-muenster.de}}
\begin{abstract}
  In this paper, we construct the Heisenberg-Virasoro algebra in the
  framework of the $\mathcal{R}(p,q)$-deformed quantum
  algebras. Moreover, the $\mathcal{R}(p,q)$-Heisenberg-Witt $n$-algebras is also investigated. Furthermore, we generalize the notion
  of the elliptic hermitian matrix models. We use the constraints to
  evaluate the $\mathcal{R}(p,q)$-differential operatos of the
  Virasoro algebra and generalize it to higher order differential
  operators. Particular cases corresponding to quantum algebras
  existing in literature are deduced.
\end{abstract}
\maketitle

%%%%%%%%%%%%%%%%%%%%%%%%%%%%
\section{Introduction}
Quantum algebras introduced by Drinfeld are used both by
mathematicians {and} physicists \cite{Drinfeld}. They {relate} to the
quantum Yang-Baxter equation which plays an important role in many
areas such as solvable lattice models, conformal field theory and
quantum integrable systems \cite{Drinfeld1}. From the mathematical
point of view, quantum algebras are Hopf algebras and generalizations
of Lie algebras \cite{CSX, CS1}.

Hounkonnou et {\it al} generalized Virasoro algebra,  relatively to their left-symmetry structure, presented related algebraic and some
hydrodynamic properties \cite{Hounkonnou:2015laa}.
The $q$-deformed Heisenberg-Virasoro algebra which is a Hom-Lie
algebra
 was constructed by Chen and
Su. The central extensions and second cohomology group were also
presented \cite{CS}.  The super $q$-deformed Virasoro $n$-algebra
for $n$ even and a toy model for the $q$-deformed Virasoro
constraints were investigated by Nedelin and Zabzine on the $q$-Virasoro constraints for a toy model \cite{NZ}.

The $\mathcal{R}(p,q)$-deformed quantum algebras and particular cases
corresponding {to} quantum algebras known in the literature were
investigated {in} \cite{HB1}.  Furthermore,
in \cite{HM}, the
$\mathcal{R}(p,q)$-deformed conformal Virasoro algebra was presented,  the
$\mathcal{R}(p,q)$-deformed Korteweg-de Vries equation for a
conformal dimension $\Delta=1,$ was derived, and  the energy-momentum
tensor induced by the ${\mathcal R}(p,q)$-quantum algebras for the
conformal dimension $\Delta=2$ was characterized.

The generalizations of Witt and Virasoro algebras, and the Korteweg-de
Vries equations from known $\mathcal{R}(p,q)$-deformed quantum
algebras were performed. The
$\mathcal{R}(p,q)$-deformed Witt $n$-algebra and the Virasoro
constraints for a toy model were constructed and determined
\cite{HMM}.

Furthermore, the $\mathcal{R}(p,q)$-deformed super Virasoro algebra
and the super $\mathcal{R}(p,q)$-Virasoro $n$-algebra ($n$ even )
were constructed. Moreover, a toy model for the super
$\mathcal{R}(p,q)$-Virasoro constraints was discussed. Particular
cases induced from the quantum algebras known in the literature were
deduced \cite{melong2022}.

Recently, the deformations of the matrix models are investigated by many authors: The definition of the deformation of the elliptic  $q,t$ matrix model was introduced by Mironov and Morozov \cite{Mironov1}.

Motivated by theses notions, the following question arises: How to generalize the Heisenberg-Virasoro algebra and  matrix models from the $\mathcal{R}(p,q)$-deformed quantum 	algebra? 

This paper is organized as follows: In section $2,$ the notion
concerning the $\mathcal{R}(p,q)$-calculus \cite{HB}, the
$\mathcal{R}(p,q)$-deformed quantum algebras \cite{HB1}, and matrix
models are recalled.  Section $3$ is devoted to construct the Heisenberg
Virasoro algebra in the framework of generalized quantum algebras
\cite{HB1}.  Section $4$ is reserved to some applications: more
precisely, the $\mathcal{R}(p,q)$-deformed Heisenberg Witt $n$-algebras is investigated. The Heisenberg Virasoro constraints {are} used
to {present} a toy model. The $\mathcal{R}(p,q)$-deformed matrix model
is determined and the generalized elliptic matrix model is
furnished. Particular cases are deduced.  We end {with} concluding
remarks in section $5.$

%%%%%%%%%%%%%%%%%%% 
\section{Preliminaries}
In this section, we fix the notations and recall some definitions and known results   useful in the sequel($\mathcal{R}(p,q)$-calculus, $\mathcal{R}(p,q)$-quantum algebras and matrix models). We start by the $\mathcal{R}(p,q)$-calculus and quantum algebra.

For that, let $ p$ and $q,$ {be} two positive real numbers such
that $ 0<q<p< 1,$ and a meromorphic function defined on
$\mathbb{C}\times\mathbb{C}$ by \cite{HB}:
\begin{equation}\label{r10}
  \mathcal{R}(s,t)= \sum_{u,v=-\eta}^{\infty}r_{uv}s^u\,t^v,
\end{equation}
where $r_{uv}$ are real
numbers {and} $\eta\in\mathbb{N}\cup\left\lbrace 0\right\rbrace,$
{such that} $\mathcal{R}(p^x,q^x)>0, \forall x\in\mathbb{N},$ and
$\mathcal{R}(1,1)=0$ by definition.  We denote by $\mathbb{D}_{R}$ the 
bidisk \begin{eqnarray*}
	\mathbb{D}_{R}
	=\left\lbrace v=(v_1,v_2)\in\mathbb{C}^2: |v_j|<R_{j} \right\rbrace,
\end{eqnarray*}
where $R$ is the convergence radius of the series (\ref{r10}) defined by Hadamard formula as follows \cite{TN}:
\begin{eqnarray*}
	\limsup_{s+t \longrightarrow \infty} \sqrt[s+t]{|r_{st}|R^s_1\,R^t_2}=1.
\end{eqnarray*}

We define the  $\mathcal{R}(p,q)$-numbers  \cite{HB}:
\begin{eqnarray}\label{rpqnumbers}
[n]_{\mathcal{R}(p,q)}:=\mathcal{R}(p^n,q^n),\quad n\in\mathbb{N},
\end{eqnarray}
the
$\mathcal{R}(p,q)$-factorials by
\begin{eqnarray*}\label{s0}
	[n]!_{\mathcal{R}(p,q)}:=\left \{
	\begin{array}{l}
		1\quad\mbox{for}\quad n=0\\
		\\
		\mathcal{R}(p,q)\cdots\mathcal{R}(p^n,q^n)\quad\mbox{for}\quad n\geq 1,
	\end{array}
	\right .
\end{eqnarray*}
and the  $\mathcal{R}(p,q)$-binomial coefficients
\begin{eqnarray*}\label{bc}
	\bigg[\begin{array}{c} m  \\ n\end{array} \bigg]_{\mathcal{R}(p,q)} := \frac{[m]!_{\mathcal{R}(p,q)}}{[n]!_{\mathcal{R}(p,q)}[m-n]!_{\mathcal{R}(p,q)}},\quad m,n\in\mathbb{N}\cup\{0\},\quad m\geq n.
\end{eqnarray*}

We denote by 
${\mathcal O}(\mathbb{D}_R)$ the set of holomorphic functions defined
on $\mathbb{D}_R$ and 
consider the following linear operators defined on  $\mathcal{O}(\mathbb{D}_{R}),$ (see \cite{HB1} for more details),
\begin{eqnarray*}
	\; P:\Psi\longmapsto P\Psi(z):&=& \Psi(pz),\\
	\;Q:\Psi\longmapsto Q\Psi(z):&=& \Psi(qz),
\end{eqnarray*}
and the $\mathcal{R}(p,q)$-derivative
\begin{eqnarray}\label{r5}
	{\mathcal D}_{\mathcal{R}( p,q)}:={\mathcal D}_{p,q}\frac{p-q}{P-Q}\mathcal{R}( P,Q)=\frac{p-q}{p^{P}-q^{Q}}\mathcal{R}(p^{P},q^{Q}){\mathcal D}_{p,q}
\end{eqnarray}
where ${\mathcal D}_{p,q}$ is the $(p,q)$-derivative:
\begin{eqnarray*}
	{\mathcal D}_{p,q}\Psi(z):=\frac{\Psi(pz)-\Psi(qz)}{z(p-q)}.
\end{eqnarray*}
The  algebra associated with the $\mathcal{R}(p,q)$-deformation is a quantum algebra, denoted $\mathcal{A}_{\mathcal{R}(p,q)},$ generated by the set of operators $\{1, A, A^{\dagger}, N\}$ satisfying the following commutation relations:
\begin{eqnarray*}
	&& \label{algN1}
	\quad A A^\dag= [N+1]_{\mathcal {R}(p,q)},\quad\quad\quad A^\dag  A = [N]_{\mathcal {R}(p,q)}.
	\cr&&\left[N,\; A\right] = - A, \qquad\qquad\quad \left[N,\;A^\dag\right] = A^\dag
\end{eqnarray*}
with the realization on  ${\mathcal O}(\mathbb{D}_R)$ given by:
\begin{eqnarray*}\label{algNa}
	A^{\dagger} := z,\qquad A:=\partial_{\mathcal {R}(p,q)}, \qquad N:= z\partial_z,
\end{eqnarray*} 
where $\partial_z:=\frac{\partial}{\partial z}$ is the  derivative on $\mathbb{C}.$

This algebra is the generalization of quantum algebras existing in the literature as follows:
\begin{itemize}
	\item[(i)]
	Taking $\mathcal{R}(x,1)=\frac{x-1}{q-1},$
	we obtain the $q$-deformed number, derivative  and the quantum algebra corresponding to the  {\bf Arick-Coon-Kuryskin algebra} \cite{AC}:
	\begin{eqnarray*}
		[n]_{q}=\frac{q^n-1}{q-1},\quad {\mathcal D}_{q}\Psi(z):=\frac{\Psi(qz)-\Psi(z)}{z(q-1)}
	\end{eqnarray*}
	and 
	\begin{eqnarray*}
		&& 
		\quad \left[N,\; A\right] = - A, \qquad\qquad\quad \left[N,\;A^\dag\right] = A^\dag.
		\cr&& A\,A^\dag-qA^\dag\,A=1 \quad \mbox{or}\quad A\,A^\dag-A^\dag\,A=q^{N}.
	\end{eqnarray*} 
	\item[(ii)]  The {\bf Biedenharn-Macfarlane algebra}\cite{B,M}, derivative and numbers can be obtained by putting $\mathcal{R}(x)=\frac{x-x^{-1}}{q-q^{-1}}:$
	\begin{eqnarray*}
		[n]_{q}=\frac{q^n-q^{-n}}{q-q^{-1}},\quad {\mathcal D}_{q}\Psi(z):=\frac{\Psi(qz)-\Psi(q^{-1}z)}{z(q-q^{-1})}
	\end{eqnarray*}
	and 
	\begin{eqnarray*}
		&& 
		\quad \left[N,\; A\right] = - A, \qquad\qquad\quad \left[N,\;A^\dag\right] = A^\dag.
		\cr&& A\,A^\dag-qA^\dag\,A=q^{-N}\quad\mbox{or}\quad A\,A^\dag-q^{-1}A^\dag\,A=q^{N},\quad q^2\neq 1 .
	\end{eqnarray*} 
	\item[(iii)]Setting $\mathcal{R}(x,y)=\frac{x-y}{p-q},$ 
	we obtain the   numbers, derivative and quantum algebra induced by the {\bf Jagannathan-Srinivasa algebra} \cite{JS}:
	\begin{eqnarray*}
		[n]_{p,q}=\frac{p^n-q^n}{p-q}, \quad {\mathcal D}_{p,q}\varPsi(z)=\frac{\varPsi(pz)-\varPsi(qz)}{z(p-q)}
	\end{eqnarray*}
	and 
	\begin{eqnarray*}
		&& 
		\quad \left[N,\; A\right] = - A, \qquad\qquad\quad \left[N,\;A^\dag\right] = A^\dag.
		\cr&& A\,A^\dag-qA^\dag\,A=p^{N} .
	\end{eqnarray*}
	\item[(iv)]Putting $	\mathcal{R}(x,y)=\frac{1-x\,y}{ (p^{-1}-q)x},$ 
	we get  the numbers, derivative, and quantum algebra from the {\bf Chakrabarty - Jagannathan algebra} \cite{Chakrabarti&Jagan}:
	\begin{eqnarray*}
		[n]_{p^{-1},q}=\frac{p^{-n}-q^n}{p^{-1}-q}, \quad
		{\mathcal D}_{p^{-1},q}\varPsi(z)=\frac{\varPsi(p^{-1}z)-\varPsi(qz)}{z(p^{-1}-q)}
	\end{eqnarray*}
	and 
	\begin{eqnarray*}
		&& 
		\quad \left[N,\; A\right] = - A, \qquad\qquad\quad \left[N,\;A^\dag\right] = A^\dag.
		\cr&& A\,A^\dag-qA^\dag\,A=p^{-N}\quad\mbox{or}\quad A\,A^\dag-q^{-1}A^\dag\,A=p^{N} .
	\end{eqnarray*}
	\item[(v)] Given $\mathcal{R}(x,y)={x\,y-1}{ (q-p^{-1})y},$ 
	we derive the numbers, derivative, and quantum algebra associated to the {\bf Hounkonnou-Ngompe  generalization of $q$-Quesne algebra} \cite{HN}:
	\begin{eqnarray*}
		[n]^Q_{p,q}=\frac{p^n-q^{-n}}{q-p^{-1}},
		\quad  
		{\mathcal D}^Q_{p,q}\varPsi(z)=\frac{\varPsi(pz)-\varPsi(q^{-1}z)}{z(q-p^{-1})},
	\end{eqnarray*}
	and \begin{eqnarray*}
		&& 
		\quad \left[N,\; A\right] = - A, \qquad\qquad\quad \left[N,\;A^\dag\right] = A^\dag.
		\cr&& p^{-1}A\,A^\dag-A^\dag\,A=q^{-N-1}\quad\mbox{or}\quad qA\,A^\dag-A^\dag\,A=p^{N+1} .
	\end{eqnarray*}
\end{itemize}

Now, we recall some notions about matrix model. We use the notation for the Schur polynomials  as polynomials of power sums $p_k=\sum_{i}z^k_i$ \cite{Mironov1}.
The Hermitean Gaussian matrix model is defined by the partition function
\begin{eqnarray*}\label{pf}
  Z_{N}(p_k):=\frac{1}{V_{N}}\int_{H_N}\,dH \exp\bigg(-\frac{1}{2}Tr\,H^2
  + \sum_{k}\frac{p_k}{k}\,Tr\,H^k\bigg),
\end{eqnarray*} 
where $H_N$ is the space of Hermitean $N\times N$ matrices, $dH$ the
Lebesgue measure and $V_{N}$ the volume of the unitary group
$U(N)$. 

The relation \eqref{pf} is a generating function of all
gauge-invariant correlators given by:
\begin{eqnarray*}
  \bigg\langle \prod_{i}Tr\,H^{k_i} \bigg\rangle:=\frac{1}{Z_{N}(0)}
  \int_{H_N}\,dH\prod_{i}Tr\,H^{k_i}\exp\big(-\frac{1}{2}Tr\,H^2\big).
\end{eqnarray*}
Integrating over {$U(N)$} in the {relation} \eqref{pf}
{gives} \cite{Mehta}
\begin{eqnarray*}\label{pf1}
Z_{N}(p_k):=\frac{1}{N!}\int_{-\infty}^{\infty}\prod_{i}\,dz_{i}\prod_{j\neq i}\big(z_{i}-z_{j}\big) \exp\bigg(-\frac{1}{2}\sum_{i}z^2_{i} + \sum_{i,k}\frac{p_k}{k}\,z_{i}^k\bigg)
\end{eqnarray*}
and 
\begin{eqnarray*}
\bigg\langle \prod_{i}\sum_{m}z_{m}^{k_i} \bigg\rangle:=\frac{1}{Z_{N}(0)}\int_{-\infty}^{\infty}\prod_{i}\,dz_{i}\prod_{j\neq i}\big(z_{i}-z_{j}\big)\bigg(\prod_{i}\sum_{m}z_{m}^{k_i}\bigg)\exp\big(-\frac{1}{2}\sum_{i}z^2_{i}\big),
\end{eqnarray*}
where $z_{i}$ are the eigenvalues of $H.$ 
\section{$\mathcal{R}(p,q)$-Heisenberg Virasoro algebra}
In this section, we construct the operators satifying the generalized
Heisenberg Witt algebra. Moreover, the central extensions are provided
and the Heisenberg Virasoro algebra is deduced in the framework of the
$\mathcal{R}(p,q)$-deformed quantum algebra. Particular cases are deduced.
\begin{Definition}
	The $\mathcal{R}(p,q)$-deformed operators $L_m$ and $I_m$ are given as follows: 
	\begin{eqnarray}\label{b}
	L_{m}\phi(z)=-z^{m}\,{\mathcal D}_{\mathcal{R}(p,q)}\phi(z),\quad \mbox{and}\quad  I_{m}\phi(z)=-(\tau\,z)^{m}\phi(z),
	\end{eqnarray}  where ${\mathcal D}_{\mathcal{R}(p,q)}$ is given by the relation \eqref{r5} and  $\tau:=\tau(p,q)$ is a parameter of deformation depending on $p$ and $q.$
\end{Definition}
\noindent Then, the $\mathcal{R}(p,q)$-Heisenberg-Witt algebra is denoted by  $\mathcal{H}_{\mathcal{R}(p,q)}:= span\{L_m, I_m/ m\in\mathbb{Z}\}.$

We introduce a family of deformations of the commutator:
$$[A,B]_{a,b}=aAB-bBA,$$ where $a$ and $b$ are referred to
as the coefficients of commutation. They can be an arbitrary complex
or real numbers. Then:
\begin{Proposition}
	The $\mathcal{R}(p,q)$-Heisenberg Witt algebra is generated by the operators (\ref{b}) obeying the following commutation relations:
	\begin{eqnarray*}
	\,\big[L_{m_1}, L_{m_2}\big]_{x,y}\phi(z)&=& [m_1-m_2]_{\mathcal{R}(p,q)}\,L_{m_1+m_2}\phi(z),\\,\big[L_{m_1}, I_{m_2}\big]_{u,v}\phi(z)&=& -[m_2]_{\mathcal{R}(p,q)}\,I_{m_1+m_2}\phi(z),\\\,\big[I_{m_1}, I_{m_2}\big]_{\mathcal{R}(p,q)}&=& 0,
	\end{eqnarray*}
	where  
	\begin{eqnarray}\label{coef}
	\left \{
	\begin{array}{l}
	x=q^{m_1-m_2}\,p^{m_1}\,\Theta_{mn}(p,q)\mbox{,}\quad y=p^{m_1}\,\Theta_{mn}(p,q),
	\\
	u=\tau^{m_1}\,p^{m_2}\mbox{,}\quad v=\tau^{m_1}\,(pq)^{m_2},\\
	\Theta_{mn}(p,q)=\frac{[m_1-m_2]_{\mathcal{R}(p,q)}}{[m_1]_{\mathcal{R}(p,q)}-(pq)^{m_1-m_2}\,[m_2]_{\mathcal{R}(p,q)}}.
	\end{array}
	\right .
	\end{eqnarray}
\end{Proposition}
\begin{proof}
Using the $\mathcal{R}(p,q)$-formula:
	\begin{align*}
		{\mathcal D}_{\mathcal{R}(p,q)}\big(f(z)g(z)\big)&= {\mathcal D}_{\mathcal{R}(p,q)}(f(z))(Pg(z)) + (Qf(z)){\mathcal D}_{\mathcal{R}(p,q)}(g(z))\\
		&= {\mathcal D}_{\mathcal{R}(p,q)}(f(z))(Qg(z)) + (Pf(z)){\mathcal D}_{\mathcal{R}(p,q)}(g(z)),
	\end{align*} we have:
	\begin{align*}
		xL_{m_1}L_{m_2}\phi(z)&=x\,z^{m_1}\,{\mathcal D}_{\mathcal{R}(p,q)}\big(z^{m_2}{\mathcal D}_{\mathcal{R}(p,q)}\phi(z)\big)\nonumber\\
		&=-x\,[m_2]_{\mathcal{R}(p,q)}\,p^{-m_2}\,L_{m_2+m_1}\,\phi(z)-xq^{m_2}\,L_{m_2+m_1}\,{\mathcal D}_{\mathcal{R}(p,q)}\phi(z)\big).
	\end{align*}
	By analogy, 
	\begin{align*}
		yL_{m_2}L_{m_1}\phi(z)&=-y\,[m_1]_{\mathcal{R}(p,q)}\,p^{-m_1}\,L_{m_2+m_1}\,\phi(z)-yq^{m_1}\,L_{m_2+m_1}\,{\mathcal D}_{\mathcal{R}(p,q)}\phi(z).
	\end{align*}
	After computation, we get:
	\begin{eqnarray*}
		\left \{
		\begin{array}{l}
			x=q^{m_1-m_2}\,p^{m_1}\,\Theta_{mn}(p,q),\\\\ y=p^{m_1}\,\Theta_{mn}(p,q),\\\\
			\Theta_{mn}(p,q)=\frac{[m_1-m_2]_{\mathcal{R}(p,q)}}{[m_1]_{\mathcal{R}(p,q)}-(pq)^{m_1-m_2}\,[m_2]_{\mathcal{R}(p,q)}}.
		\end{array}
		\right .
	\end{eqnarray*}
	Moreover, we use the same technique to obtain 
	$u=\tau^{m_1}\,p^{m_2}$ and $v=\tau^{m_1}\,(pq)^{m_2}.$
\end{proof}
\begin{Remark}
	There exist another way to construct the $\mathcal{R}(p,q)$-Heisenberg Witt algebra. Here, we consider $\mathcal{H}_{\mathcal{R}(p,q)}$ be a non associative algebra with basis $\{z^m\,{\mathcal D}^{s}_{\mathcal{R}(p,q)}/ m\in\mathbb{Z}, s\in\mathbb{N}\}$ and  defined the following product:
	\begin{eqnarray*}
	\big(z^{m_1}{\mathcal D}^{s_1}_{\mathcal{R}(p,q)} \big)\circ \big(z^{m_2}{\mathcal D}^{s_2}_{\mathcal{R}(p,q)}\big):=z^{m_1+m_2}\sum_{i=0}^{s_1}\binom{s_1}{i}\,[m_2]^{i}_{\mathcal{R}(p,q)}\,{\mathcal D}^{s_1+s_2-i}_{\mathcal{R}(p,q)},
	\end{eqnarray*}
	with $(m_1,m_2)\in\mathbb{Z}\times\mathbb{Z}$ and $(s_1,s_2)\in\mathbb{N}\times\mathbb{N}.$
	
	Therefore,  the operators $L_m$ and $I_m$ satisfy the commutation relations presented by:
	\begin{eqnarray}
	\,\big[L_{m_1}, L_{m_2}\big]_{\mathcal{R}(p,q)}\phi(z)&=& [m_1-m_2]_{\mathcal{R}(p,q)}\,L_{m_1+m_2}\phi(z),\label{hwa}\\\,\big[L_{m_1}, I_{m_2}\big]_{\mathcal{R}(p,q)}\phi(z)&=& -\tau^{-m_1}\,[m_2]_{\mathcal{R}(p,q)}\,I_{m_1+m_2}\phi(z),\label{hwb}\\\,\big[I_{m_1}, I_{m_2}\big]_{\mathcal{R}(p,q)}\phi(z)&=& 0\label{hwc}.
	\end{eqnarray}
\end{Remark}
\begin{Definition}
 A Hom-Lie algebra is a vector space with skew symmetric bracket and
	generalised Jacobi identity
	$[\alpha(x),[y,z]]+[\alpha(y),[z,x]]+[\alpha(z),[x,y]]=0$
	for an endomorphism $\alpha$.
	\end{Definition}
\begin{Definition}
	A $\mathcal{R}(p,q)$-deformed $2$-cocycle on $\mathcal{H}_{\mathcal{R}(p,q)}$ is a bilinear function $\Psi: \mathcal{H}_{\mathcal{R}(p,q)} \times \mathcal{H}_{\mathcal{R}(p,q)} \longrightarrow \mathbb{C}$ verifying the following conditions:
	\begin{eqnarray}
	\,\Psi(x,y)&=&-\Psi(y,x),\label{a2cocy}\\
	\,\Psi([x,y]_{\mathcal{R}(p,q)}, \alpha(z))&=&\Psi(\alpha(x),[y,z]_{\mathcal{R}(p,q)}) + \Psi([x,z]_{\mathcal{R}(p,q)}, \alpha(y))\label{b2cocy},
	\end{eqnarray}
	where $x,y,z \in \mathcal{H}_{\mathcal{R}(p,q)}$ and \begin{eqnarray*}
		\alpha(L_m)=\frac{[2\,m]_{\mathcal{R}(p,q)}}{[m]_{\mathcal{R}(p,q)}}L_m\quad \mbox{and}\quad \alpha(I_m)=\frac{[2\,m]_{\mathcal{R}(p,q)}}{[m]_{\mathcal{R}(p,q)}}I_m .
	\end{eqnarray*}
\end{Definition}

Note that the $\mathcal{R}(p,q)$-numbers \eqref{rpqnumbers} can be rewritten in the form \cite{HMM}:
\begin{eqnarray*}
	[n]_{\mathcal
		{R}(p,q)}=\frac{\epsilon^n_1-\epsilon^n_2}{\epsilon_1-\epsilon_2},\quad \epsilon_1\neq \epsilon_2,
\end{eqnarray*}
where  $\epsilon_i, i \in\{1,2\},$ are  the structure functions
depending on the deformation parameters $p$ and $q.$
\begin{Lemma}\cite{HMM}
	The $\mathcal{R}(p,q)$-Jacobi identity is given by:
	\begin{eqnarray}\label{rpqJI}
	\displaystyle
	\sum_{(i,j,l)\in\mathcal{C}(n,m,k)} (\frac{1}{\epsilon_1\epsilon_2})^{-l}\frac{[2i]_{\mathcal{R}(p,q)}}{[i]_{\mathcal{R}(p,q)}}\big[L_i, \big[L_j , L_l\big]_{\mathcal{R}(p,q)}\big]_{\mathcal{R}(p,q)} =0,
	\end{eqnarray}
	where  $n$, $m$ and $k$ are natural numbers, and  $\mathcal{C}(n,m,k)$ refers  to the
	cyclic permutation of $(n,m,k)$.
\end{Lemma}

Let us now present the Heisenberg Virasoro algebra from the $\mathcal{R}(p,q)$- quantum algebra. It's an extension of the  $\mathcal{R}(p,q)$-Heisenberg Witt algebra given by  (\ref{hwa}), (\ref{hwb}), and (\ref{hwc}). The central extension of the relation (\ref{hwa}) is well known in our previous work as follows \cite{HMM}: 
\begin{eqnarray*}\label{Rct}
	C_{\mathcal{R}(p,q)}(n)=
	C(p,q)\big(\frac{q}{p}\big)^{-n}\frac{[n]_{\mathcal{R}(p,q)}} {6[2\,n]_{\mathcal{R}(p,q)}}\,[n-1]_{\mathcal{R}(p,q)}\,[n]_{\mathcal{R}(p,q)}\,[n+1]_{\mathcal{R}(p,q)},
\end{eqnarray*}
where
$C(p,q)$ is an arbitrary function of $(p,q).$

From the relations (\ref{a2cocy}), (\ref{b2cocy}), and (\ref{rpqJI}), we can obtain:
\begin{eqnarray}\label{cli}
	C_{LI}(m_1)=
	C_{LI}(p,q)\big(\frac{q}{p}\big)^{-m_1}\,\frac{2[m_1]_{\mathcal{R}(p,q)}} {[2\,m_1]_{\mathcal{R}(p,q)}}\,[m_1]_{\mathcal{R}(p,q)}\,[m_1+1]_{\mathcal{R}(p,q)},
\end{eqnarray}
and 
\begin{eqnarray}\label{ci}
	C_{I}(m_1)=
	C_{I}(p,q)\big(\frac{q}{p}\big)^{-m_1}\,\frac{2[m_1]_{\mathcal{R}(p,q)}} {[2\,m_1]_{\mathcal{R}(p,q)}}\,[m_1]_{\mathcal{R}(p,q)}.
\end{eqnarray}

Then, the $\mathcal{R}(p,q)$-deformed Heisenberg-Virasoro algebra $\bar{\mathcal{H}}_{\mathcal{R}(p,q)}:=span\{\bar{L}_m,\bar{I}_m/ m\in\mathbb{Z}\}.$ 
\begin{Proposition}
  The $\mathcal{R}(p,q)$-deformed Heisenberg Virasoro algebra is governed by the following commutation relations:
		\begin{eqnarray*}
		\,\big[\bar{L}_{m_1}, \bar{L}_{m_2}\big]_{x,y}\phi(z)&=& [m_1-m_2]_{\mathcal{R}(p,q)}\,\bar{L}_{m_1+m_2}\phi(z) + C_{L}(m_1)\,\delta_{m_1+m_2,0}\label{hva},\\\,\big[\bar{L}_{m_1}, \bar{I}_{m_2}\big]_{u,v}\phi(z)&=& -[m_2]_{\mathcal{R}(p,q)}\,\bar{I}_{m_1+m_2}\phi(z)+ C_{LI}(m_1)	\delta_{m_1+m_2,0}\label{hvb},\\\,\big[\bar{I}_{m_1}, \bar{I}_{m_2}\big]_{\mathcal{R}(p,q)}\phi(z)&=& C_{I}(p,q)\big(\frac{q}{p}\big)^{m_1}\,\frac{2[m_1]_{\mathcal{R}(p,q)}} {[2\,m_1]_{\mathcal{R}(p,q)}}\,[m_1]_{\mathcal{R}(p,q)}\delta_{m_1+m_2,0}\label{hvc},
		\end{eqnarray*}
		where 
		\begin{align*}
			\,C_{L}(m_1)&=C_{L}(p,q)\big(\frac{q}{p}\big)^{-m_1}\frac{[m_1]_{\mathcal{R}(p,q)}} {6[2\,m_1]_{\mathcal{R}(p,q)}}\,[m_1-1]_{\mathcal{R}(p,q)}\,[m_1]_{\mathcal{R}(p,q)}\,[m_1+1]_{\mathcal{R}(p,q)},\\
			\,C_{LI}(m_1)&=C_{LI}(p,q)\big(\frac{q}{p}\big)^{-m_1}\,\frac{2[m_1]_{\mathcal{R}(p,q)}} {[2\,m_1]_{\mathcal{R}(p,q)}}\,[m_1]_{\mathcal{R}(p,q)}\,[m_1+1]_{\mathcal{R}(p,q)},
		\end{align*}
	and $x,$ $y,$ $u,$ and $v$ are given by the relation \eqref{coef}.
\end{Proposition}

\begin{Remark} It is necessary to derive particular cases of Heisenberg Virasoro algebra induced by the deformed quantum algebra known in the literature.
	\begin{enumerate}
		\item [(i)]The $q$-operators $\bar{L}_m=-z^m\,{\mathcal D}_q$ and $\bar{I}_m=-q^m\,z^m$ satisfy the $q$-Heisenberg Virasoro algebra with the commutation relations:
			\begin{align*}
				\big[\bar{L}_{m_1}, \bar{L}_{m_2}\big]_{x,y}\phi(z)&= [m_1-m_2]_{q}\bar{L}_{m_1+m_2}\phi(z)\\ &+ \frac{C_{L}(q)q^{-m_1}} {12(1+q^{m_1})}[m_1-1]_{q}[m_1]_{q}[m_1+1]_{q}\delta_{m_1+m_2,0}\\\big[\bar{L}_{m_1}, \bar{I}_{m_2}\big]_{u,v}\phi(z)&= -[m_2]_{q}\,\bar{I}_{m_1+m_2}\phi(z)\\&+ \frac{2C_{LI}(q)\,q^{-m_1}} {1+q^{m_1}}[m_1]_{q}[m_1+1]_{q}	\delta_{m_1+m_2,0}\\\big[\bar{I}_{m_1}, \bar{I}_{m_2}\big]_{q}\phi(z)&= C_{I}(q)\,q^{m_1}\,\frac{2[m_1]_{q}} {[2\,m_1]_{q}}[m_1]_{q}\delta_{m_1+m_2,0},
			\end{align*}
			where  
			\begin{eqnarray*}
				\left \{
				\begin{array}{l}
					x=q^{m_1-m_2}\,\Theta_{mn}(q)\mbox{,}\quad y=\Theta_{mn}(q),
					\\
					u=q^{m_1}\mbox{,}\quad v=q^{m_1-m_2},\\
					\Theta_{mn}(q)=\frac{[m_1-m_2]_{q}}{[m_1]_{q}-q^{m_1-m_2}\,[m_2]_{q}}.
				\end{array}
				\right .
			\end{eqnarray*}
		\item[(ii)] The $q$-Heisenberg Virasoro algebra generated by the $q$-operators $\bar{L}_m=-z^m\,{\mathcal D}_q$ and $\bar{I}_m=-q^{2m}\,z^m$ obeys the following commutation relations:
		\begin{align*}
			\big[\bar{L}_{m_1}, \bar{L}_{m_2}\big]_{x,y}\phi(z)&= [m_1-m_2]_{q}\bar{L}_{m_1+m_2}\phi(z)\\ &+ \frac{C_{L}(q)q^{-2\,m_1}} {12(q^{m_1}+q^{-m_1})}[m_1-1]_{q}[m_1]_{q}[m_1+1]_{q}\delta_{m_1+m_2,0}\\\big[\bar{L}_{m_1}, \bar{I}_{m_2}\big]_{u,v}\phi(z)&= -[m_2]_{q}\bar{I}_{m_1+m_2}\phi(z)\\&+ \frac{2C_{LI}(q)\,q^{-2\,m_1}} {q^{m_1}+q^{-m_1}}[m_1]_{q}[m_1+1]_{q}	\delta_{m_1+m_2,0}\\\big[\bar{I}_{m_1}, \bar{I}_{m_2}\big]_{q}\phi(z)&= C_{I}(q)\,q^{2\,m_1}\,\frac{2[m_1]_{q}} {[2\,m_1]_{q}}[m_1]_{q}\,\delta_{m_1+m_2,0},
		\end{align*}
		where  
		\begin{eqnarray*}
			\left \{
			\begin{array}{l}
				x=q^{2m_1-m_2}\,\Theta_{mn}(q)\mbox{,}\quad y=q^{m_1}\,\Theta_{mn}(q),
				\\
				u=q^{2m_1+m_2}\mbox{,}\quad v=q^{2m_1},\\
				\Theta_{mn}(q)=\frac{[m_1-m_2]_{q}}{[m_1]_{q}-[m_2]_{q}}.
			\end{array}
			\right .
		\end{eqnarray*}
		\item[(iii)]The  Heisenberg- Virasoro algebra induced by the {\bf Chakrabarty - Jaganathan algebra} is generated by the $(p,q)$- operators $\bar{L}_m=-z^m\,{\mathcal D}_{p,q}$ and $\bar{I}_m=-\big(\frac{q}{p}\big)^{m}\,z^m$ verifying the commutation relations:
			\begin{align*}
				\big[\bar{L}_{m_1}, \bar{L}_{m_2}\big]_{x,y}\phi(z)&= [m_1-m_2]_{p,q}\bar{L}_{m_1+m_2}\phi(z)\\ &+ \frac{C_{L}(p,q)q^{-2\,m_1}} {12(p^{m_1}+q^{m_1})}[m_1-1]_{p,q}[m_1]_{p,q}[m_1+1]_{p,q}\delta_{m_1+m_2,0}\\\big[\bar{L}_{m_1}, \bar{I}_{m_2}\big]_{u,v}\phi(z)&= -q^{m_1}\,[m_2]_{p,q}\bar{I}_{m_1+m_2}\phi(z)\\&+ \frac{2C_{LI}(p,q)\,q^{-2\,m_1}} {p^{m_1}+q^{m_1}}[m_1]_{p,q}[m_1+1]_{p,q}	\delta_{m_1+m_2,0}\\\big[\bar{I}_{m_1}, \bar{I}_{m_2}\big]_{p,q}\phi(z)&= C_{I}(q)\,q^{2\,m_1}\,\frac{2[m_1]_{p,q}} {[2\,m_1]_{p,q}}[m_1]_{p,q}\,\delta_{m_1+m_2,0},
			\end{align*}
			where  
			\begin{eqnarray*}
				\left \{
				\begin{array}{l}
					x=q^{m_1-m_2}\,p^{m_1}\,\Theta_{m_1m_2}(p,q)\mbox{,}\quad y=p^{m_1}\,\Theta_{m_1m_2}(p,q),
					\\
					u=q^{m_1}\,p^{m_2-m_1}\mbox{,}\quad v=q^{m_1+m_2}\,p^{m_2-m_1},\\
					\Theta_{mn}(p,q)=\frac{[m_1-m_2]_{p,q}}{[m_1]_{p,q}-(pq)^{m_1-m_2}\,[m_2]_{p,q}}.
				\end{array}
				\right .
			\end{eqnarray*}
		\item[(iv)]The  Heisenberg Virasoro algebra associated to the {\bf generalized $q$-Quesne algebra} is governed by the  operators $\bar{L}_m=-z^m\,{\mathcal D}^Q_{p,q}$ and $\bar{I}_m=-\big(\frac{q}{p}\big)^{m}\,z^m$ obeying the commutation relations:
			\begin{align*}
				\big[\bar{L}_{m_1}, \bar{L}_{m_2}\big]_{x,y}\phi(z)&= [m_1-m_2]^Q_{p,q}\bar{L}_{m_1+m_2}\phi(z)\\ &+ \frac{C_{L}(p,q)q^{-2\,m_1}} {12(q^{m_1}+q^{-m_1})}[m_1-1]_{q}[m_1]^Q_{p,q}[m_1+1]^Q_{p,q}\delta_{m_1+m_2,0}\\\big[\bar{L}_{m_1}, \bar{I}_{m_2}\big]_{u,v}\phi(z)&= -q^{m_1}\,[m_2]^Q_{p,q}\bar{I}_{m_1+m_2}\phi(z)\\&+ \frac{2C_{LI}(q)\,q^{-2\,m_1}} {q^{m_1}+q^{-m_1}}[m_1]^Q_{p,q}[m_1+1]^Q_{p,q}	\delta_{m_1+m_2,0}\\\big[\bar{I}_{m_1}, \bar{I}_{m_2}\big]^Q_{p,q}\phi(z)&= C_{I}(q)\,q^{2\,m_1}\,\frac{2[m_1]^Q_{p,q}} {[2\,m_1]^Q_{p,q}}[m_1]^Q_{p,q}\,\delta_{m_1+m_2,0},
			\end{align*}
			where  
			\begin{eqnarray*}
				\left \{
				\begin{array}{l}
					x=q^{-m_1+m_2}\,p^{m_1}\,\Theta^Q_{m_1m_2}(p,q)\mbox{,}\quad y=p^{m_1}\,\Theta^Q_{m_1m_2}(p,q),
					\\
					u=q^{-m_1}\,p^{m_2-m_1}\mbox{,}\quad v=q^{-m_1+m_2}\,p^{m_2+m_1},\\
					\Theta^Q_{m_1m_2}(p,q)=\frac{[m_1-m_2]^Q_{p,q}}{[m_1]^Q_{p,q}-(pq)^{m_1-m_2}\,[m_2]^Q_{p,q}}.
				\end{array}
				\right .
			\end{eqnarray*}
	\end{enumerate}
\end{Remark}
\section{Applications}
This section is reserved to some applications of the generalized Heisenberg Virasoro algebra. Presicely, we study the generalized Heisenberg Witt $n$-algebras, a toy model for the Heisenberg Virasoro algebra, the $\mathcal{R}(p,q)$-deformed matrix models, and the elliptic generalized matrix models.
\subsection{$\mathcal{R}(p,q)$-Heisenberg Witt $n$-algebras}
We  construct the Heisenberg Witt $n$-algebras from the $\mathcal{R}(p,q)$-deformed quantum algebras \cite{HB1}. Particular cases are deduced. We consider the following relation for the $\mathcal{R}(p,q)$-deformed derivative:
\begin{eqnarray}\label{rpqder}
{\mathcal D}_{{\mathcal R}(p,q)}:=\frac{1}{z}\,[z\partial_z]_{\mathcal{R}(p,q)}
\end{eqnarray}
and  the operators given by:
\begin{equation*}
{\mathbb T}^{{\mathcal R}(p^{a},q^{a})}_m\phi(z):=-z^{m+1}\,{\mathcal D}_{{\mathcal R}(p^{a},q^{a})}\phi(z),
\end{equation*}
where ${\mathcal D}_{{\mathcal R}(p^{a},q^{a})}$ is the ${\mathcal R}(p,q)$-deformed derivative defined as:
\begin{equation*}
	{\mathcal D}_{{\mathcal R}(p^{a},q^{a})}\big(\phi(z)\big):=\frac{p^{a}-q^{a}}{p^{a\,P}-q^{a\,Q}}{\mathcal R}(p^{a\,P},q^{a\,Q})\,\frac{\phi(p^{a}z)-\phi(q^{a}z)}{p^{a}-q^{a}}.
\end{equation*}
Then, from the relation (\ref{rpqder}),
 the $\mathcal{R}(p,q)$-deformed operators can be rewritten as follows:
\begin{eqnarray}\label{to1}
{\mathbb T}^{{\mathcal R}(p^{a},q^{a})}_m \phi(z)=-[z\partial_z-m]_{\mathcal{R}(p^{a},q^{a})}\,z^{m}\phi(z).
\end{eqnarray}
Moreover, we define the second operators as follows:
\begin{eqnarray}\label{to2}
\mathbb{I}^{{\mathcal R}(p^{a},q^{a})}_m\phi(z):=- \tau^{a}\,z^m\,\phi(z).
\end{eqnarray}
\begin{Proposition}
	The deformed operators (\ref{to1}) and \eqref{to2} satisfy the product relations:
	 \begin{align*}\label{pre}
	{\mathbb T}^{{\mathcal R}(p^{a},q^{a})}_n.{\mathbb T}^{{\mathcal R}(p^{b},q^{b})}_m&=
-{\big(\epsilon^{a+b}_1-\epsilon^{a+b}_2\big)\epsilon^{m\,a}_1\over \big(\epsilon^{a}_1-\epsilon^{a}_2\big)\big(\epsilon^{b}_1-\epsilon^{b}_2\big)}\,{\mathbb T}^{{\mathcal R}(p^{a+b},q^{a+b})}_{m+n} + {\epsilon^{(z\partial_z-n)a}_2\over \epsilon^{b}_1-\epsilon^{b}_2}\, {\mathbb T}^{{\mathcal R}(p^{a},q^{a})}_{m+n}\\ &+ {\epsilon^{m\,a}_1\,\epsilon^{(z\partial_z-m-n)b}_2\over \epsilon^{a}_1-\epsilon^{a}_2}\,{\mathbb T}^{{\mathcal R}(p^{b},q^{b})}_{m+n}
	\end{align*}
and 
\begin{align*}
\mathbb{T}^{{\mathcal R}(p^{a},q^{a})}_n.\mathbb{I}^{{\mathcal R}(p^{b},q^{b})}_m&= \frac{1}{\epsilon^{a}_1-\epsilon^{a}_2}\bigg\{ \tau^{-a}\epsilon^{a(z\partial_z-n)}_2\,\mathbb{I}^{{\mathcal R}(p^{a+b},q^{a+b})}_{n+m}-\big(\epsilon^{a(z\partial_z-n)}_1-1\big)\mathbb{I}^{{\mathcal R}(p^{b},q^{b})}_{n+m}\\ &- \tau^{b-a}\,\mathbb{I}^{{\mathcal R}(p^{a},q^{a})}_{n+m}\bigg\}.
\end{align*}
Furthermore,  the following commutation relations hold:
		\begin{eqnarray}\label{crto}
		\Big[{\mathbb T}^{{\mathcal R}(p^{a},q^{a})}_n, {\mathbb T}^{{\mathcal R}(p^{b},q^{b})}_m\Big]&=&{\big(\epsilon^{a+b}_1-\epsilon^{a+b}_2\big)\big(\epsilon^{n\,b}_1-\epsilon^{m\,a}_1\big)\over \big(\epsilon^{a}_1-\epsilon^{a}_2\big)\big(\epsilon^{b}_1-\epsilon^{b}_2\big)}{\mathbb T}^{{\mathcal R}(p^{a+b},q^{a+b})}_{m+n}\nonumber\\ &-&{\epsilon^{(z\partial_z-m-n)a}_2\big(\epsilon^{n\,b}_1-\epsilon^{m\,a}_2\big)\over \epsilon^{b}_1-\epsilon^{b}_2}{\mathbb T}^{{\mathcal R}(p^{a},q^{a})}_{m+n}\nonumber\\&+& {\epsilon^{(z\partial_z-m-n)b}_2\big(\epsilon^{m\,a}_1-\epsilon^{n\,b}_2\big)\over \epsilon^{a}_1-\epsilon^{a}_2}{\mathbb T}^{{\mathcal R}(p^{b},q^{b})}_{m+n}
		\end{eqnarray}
and 
	\begin{eqnarray}\label{crto2}
	\Big[{\mathbb T}^{{\mathcal R}(p^{a},q^{a})}_n, {\mathbb I}^{{\mathcal R}(p^{b},q^{b})}_m\Big]&=&\frac{1}{\epsilon^{a}_1-\epsilon^{a}_2}\bigg\{\tau^{-a}\epsilon^{a(z\partial_z-n)}_2\big(1-\epsilon^{-m\,a}_2\big)\mathbb{I}^{{\mathcal R}(p^{a+b},q^{a+b})}_{n+m}\nonumber\\&-&\epsilon^{a(z\partial_z-n)}_1\big(\epsilon^{-m\,a}_1 - 1\big)\mathbb{I}^{{\mathcal R}(p^{b},q^{b})}_{n+m} \bigg\}
	.\end{eqnarray}
\end{Proposition}
\begin{proof}
	By simple computation. 
\end{proof}

Putting $a=b=1,$ we obtain, respectively, 
	\begin{align*}
		\Big[{\mathbb T}^{{\mathcal R}(p,q)}_n, {\mathbb T}^{{\mathcal R}(p,q)}_m\Big]&={\big(\epsilon^{n}_1-\epsilon^{m}_1\big)\over \big(\epsilon_1-\epsilon_2\big)}[2]_{{\mathcal R}(p,q)}{\mathbb T}^{{\mathcal R}(p^{2},q^{2})}_{m+n}\\&-\epsilon^{z\partial_z-m-n}_2\Big([n]_{{\mathcal R}(p,q)}+[m]_{{\mathcal R}(p,q)}\Big) {\mathbb T}^{{\mathcal R}(p,q)}_{m+n}
	\end{align*}
and 
	\begin{align*}
	\Big[{\mathbb T}^{{\mathcal R}(p,q)}_n, {\mathbb I}^{{\mathcal R}(p,q)}_m\Big]&=\frac{1}{\epsilon_1-\epsilon_2}\bigg\{\tau^{-1}\epsilon^{(z\partial_z-n)}_2\big(1-\epsilon^{-m}_2\big)\mathbb{I}^{{\mathcal R}(p^2,q^2)}_{n+m}\\&-\epsilon^{(z\partial_z-n)}_1\big(\epsilon^{-m}_1 - 1\big)\mathbb{I}^{{\mathcal R}(p,q)}_{n+m} \bigg\}
	.\end{align*}

We consider the $n$-brackets defined by:
\begin{eqnarray}\label{nbracket}
	\Big[{\mathbb T}^{{\mathcal R}(p^{a_1},q^{a_1})}_{m_1},\cdots,{\mathbb T}^{{\mathcal R}(p^{a_n},q^{a_n})}_{m_n}
	\Big]:=\Gamma^{i_1 \cdots i_n}_{1 \cdots n}\,{\mathbb T}^{{\mathcal R}(p^{a_{i_1}},q^{a_{i_1}})}_{m_{i_1}} \cdots {\mathbb T}^{{\mathcal R}(p^{a_{i_n}},q^{a_{i_n}})}_{m_{i_n}},
\end{eqnarray}
and 
	\begin{eqnarray}\label{rnb2}
	\bigg[{\mathbb T}^{{\mathcal R}(p^{a},q^{a})}_{m_1},\cdots, {\mathbb I}^{{\mathcal R}(p^{a},q^{a})}_{m_n}\bigg]&:=&
	\sum_{j=0}^{n-1}(-1)^{n-1+j}\Gamma^{i_1\ldots i_{n-1}}_{12\cdots n-1}{\mathbb T}^{{\mathcal R}(p^{a},q^{a})}_{m_{i_1}}\ldots {\mathbb T}^{{\mathcal R}(p^{a},q^{a})}_{m_{i_j}}\nonumber\\&\times& 
	{\mathbb I}^{{\mathcal R}(p^{a},q^{a})}_{m_{n}}{\mathbb T}^{{\mathcal R}(p^{a},q^{a})}_{m_{i_{j+1}}}\cdots {\mathbb T}^{{\mathcal R}(p^{a},q^{a})}_{m_{i_{n-1}}},
	\end{eqnarray}
where $\Gamma^{i_1 \cdots i_n}_{1 \cdots n}$ is the L\'evi-Civit\'a symbol given by:
\begin{eqnarray*}
	\Gamma^{j_1 \cdots j_p}_{i_1 \cdots i_p}= det\left( \begin{array} {ccc}
		\delta^{j_1}_{i_1} &\cdots&  \delta^{j_1}_{i_p}   \\ 
		\vdots && \vdots \\
		\delta^{j_p}_{i_1} & \cdots& \delta^{j_p}_{i_p}
		\end {array} \right) .
	\end{eqnarray*}
	We are interested on the case with the same ${\mathcal R}(p^{a},q^{a}).$ Then, 
	\begin{eqnarray*}
		\Big[{\mathbb T}^{{\mathcal R}(p^{a},q^{a})}_{m_1},\cdots,{\mathbb T}^{{\mathcal R}(p^{a},q^{a})}_{m_n}
		\Big]=\Gamma^{1\cdots n}_{1\cdots n}\,{\mathbb T}^{{\mathcal R}(p^{a},q^{a})}_{m_{1}}\cdots {\mathbb T}^{{\mathcal R}(p^{a},q^{a})}_{m_{n}}.
	\end{eqnarray*}
	Putting $a=b$ in the relations \eqref{crto} and \eqref{crto2}, we obtain:
		\begin{align*}
		\Big[{\mathbb T}^{{\mathcal R}(p^{a},q^{a})}_n, {\mathbb T}^{{\mathcal R}(p^{a},q^{a})}_m\Big]&=\frac{\big(\epsilon^{n\,a}_1-\epsilon^{m\,a}_1\big)}{ \epsilon^{a}_1-\epsilon^{a}_2}\,[2]_{\mathcal{R}(p^{a},q^{a})}{\mathbb T}^{{\mathcal R}(p^{2\,a},q^{2\,a})}_{m+n}\\ 
		&+ \frac{\epsilon^{(z\partial_z-m-n)a}_2}{\epsilon^{a}_1-\epsilon^{a}_2}\big(\epsilon^{m\,a}_1-\epsilon^{n\,a}_1+\epsilon^{m\,a}_2-\epsilon^{n\,a}_2\big) {\mathbb T}^{{\mathcal R}(p^{a},q^{a})}_{m+n}
		\end{align*}
	and 
		\begin{align*}
		\Big[{\mathbb T}^{{\mathcal R}(p^{a},q^{a})}_n, {\mathbb I}^{{\mathcal R}(p^{a},q^{a})}_m\Big]&=\frac{1}{\epsilon^{a}_1-\epsilon^{a}_2}\bigg\{\tau^{-a}\epsilon^{a(z\partial_z-n)}_2\big(1-\epsilon^{-m\,a}_2\big)\mathbb{I}^{{\mathcal R}(p^{2\,a},q^{2\,a})}_{n+m}\\&-\epsilon^{a(z\partial_z-n)}_1\big(\epsilon^{-m\,a}_1 - 1\big)\mathbb{I}^{{\mathcal R}(p^{a},q^{a})}_{n+m} \bigg\}.
		\end{align*}
	After computation, the $n$-brackets \eqref{nbracket} and \eqref{rnb2} can  be reduced in the form as  follows:
		\begin{align*}
		\Big[{\mathbb T}^{{\mathcal R}(p^{a},q^{a})}_{m_1},\cdots, {\mathbb T}^{{\mathcal R}(p^{a},q^{a})}_{m_n}\Big]&=\frac{1}{ \big(\epsilon^{a}_1-\epsilon^{a}_2\big)^{n-1}}\Big( M^n_{a}[n]_{{\mathcal R}(p^{a},q^{a})}{\mathbb T}^{{\mathcal R}(p^{n\,a},q^{n\,a})}_{m_1+\cdots+m_n}\\ &- {[n-1]_{{\mathcal R}(p^{a},q^{a})}\over  \epsilon^{-a\big(\sum_{l=1}^{n}z\partial_z-m_l\big)}_2}\big(M^n_{a}+ W^n_{a}\big){\mathbb T}^{{\mathcal R}(p^{(n-1)a},q^{(n-1)a})}_{m_1+\cdots+m_n}\Big),
		\end{align*}
	and 
	\begin{align}
	\bigg[{\mathbb T}^{{\mathcal R}(p^{a},q^{a})}_{m_1},{\mathbb T}^{{\mathcal R}(p^{a},q^{a})}_{m_2},\cdots, {\mathbb I}^{{\mathcal R}(p^{a},q^{a})}_{m_n}\bigg]&= \frac{1}{\big(\epsilon^{\alpha}_1-\epsilon^{a}_2\big)^{n-1}}\bigg\{F^{a}_n\,\mathbb{I}^{{\mathcal R}(p^{n\,a},q^{n\,a})}_{m_1+\cdots+m_n}\nonumber\\&-R^{a}_n\,\mathbb{I}^{{\mathcal R}(p^{(n-1)a},q^{(n-1)a})}_{m_1+\cdots+m_n} \bigg\},
	\end{align}
	where 
		\begin{align*}
			M^n_{\alpha}&=
			\epsilon^{a(n-1)\sum_{s=1}^{n}m_s}_1\,\Big(\big(\epsilon^{a}_1-\epsilon^{a}_2\big)^{n\choose 2}\prod_{1\leq j < k \leq n}\Big([-m_j]_{{\mathcal R}(p^{a},q^{a})}-[-m_k]_{{\mathcal R}(p^{a},q^{a})}\Big)\\&+\prod_{1\leq j < k \leq n}\Big(\epsilon^{-a\,m_j}_1-\epsilon^{-a\,m_k}_1\Big)\Big),
		\end{align*}
		\begin{align*}
			W^{n}_{a}
			&=\epsilon^{a(n-1)\sum_{s=1}^{n}m_s}_2\Big(\big(\epsilon^{a}_1-\epsilon^{a}_2\big)^{n\choose 2}\prod_{1\leq j < k \leq n}\Big([-m_j]_{{\mathcal R}(p^{a},q^{a})}-[-m_k]_{{\mathcal R}(p^{a},q^{a})}\Big)\\&+(-1)^{n-1}\prod_{1\leq j < k \leq n}\Big(\epsilon^{-a\,m_j}_1-\epsilon^{-a\,m_k}_1\Big)\Big),
		\end{align*}
	\begin{eqnarray*}
	F^{a}_n = \tau^{-a}\epsilon^{a\sum_{s=1}^{n-1}(z\partial_z-m_s)}_2\big(1-\epsilon^{-a\sum_{s=1}^{n-1}m_s}_2\big)
	\end{eqnarray*}
and 
\begin{eqnarray*}
R^{a}_n=\epsilon^{a\sum_{s=1}^{n-1}(z\partial_z-m_s)}_1\big(\epsilon^{-a\sum_{s=1}^{n-1}m_s}_1 - 1\big).
\end{eqnarray*}
\begin{Remark}
	Taking $n=3,$ we obtain the ${\mathcal R}(p,q)$-Heisenberg Witt $3$-algebra:
	\begin{align*}
		\Big[{\mathbb T}^{{\mathcal R}(p^{a},q^{a})}_{m_1},{\mathbb T}^{{\mathcal R}(p^{a},q^{a})}_{m_2}, {\mathbb T}^{{\mathcal R}(p^{a},q^{a})}_{m_3}\Big]&={1\over \big(\epsilon^{a}_1-\epsilon^{a}_2\big)^{2}}\Big( M^3_{a}[3]_{{\mathcal R}(p^{a},q^{a})}{\mathbb T}^{{\mathcal R}(p^{3\,a},q^{3\,a})}_{m_1+m_2+m_3}\\ &- {[2]_{{\mathcal R}(p^{a},q^{a})}\over  \epsilon^{a\big(\sum_{l=1}^{3}z\partial_z-m_l\big)}_2}\big(M^3_{a}+ W^3_{a}\big){\mathbb T}^{{\mathcal R}(p^{2\,a},q^{2\,a})}_{m_1+m_2+m_3}\Big),
	\end{align*}
	where 
		\begin{align*}
			M^3_{a}&=
			\epsilon^{2\,a\sum_{s=1}^{3}m_s}_1\,\Big(\big(\epsilon^{a}_1-\epsilon^{a}_2\big)^{3\choose 2}\prod_{1\leq j < k \leq 3}\Big([-m_j]_{{\mathcal R}(p^{a},q^{a})}-[-m_k]_{{\mathcal R}(p^{a},q^{a})}\Big)\\&+\prod_{1\leq j < k \leq 3}\Big(\epsilon^{-a\,m_j}_2-\epsilon^{-a\,m_k}_2\Big)\Big)
		\end{align*}
		and 
		\begin{align*}
			W^{3}_{a}
			&=\epsilon^{2\,a\sum_{s=1}^{3}m_s}_2\Big(\big(\epsilon^{a}_1-\epsilon^{a}_2\big)^{n\choose 2}\prod_{1\leq j < k \leq 3}\Big([-m_j]_{{\mathcal R}(p^{a},q^{a})}-[-m_k]_{{\mathcal R}(p^{a},q^{a})}\Big)\\&+\prod_{1\leq j < k \leq 3}\Big(\epsilon^{-a\,m_j}_1-\epsilon^{-a\,m_k}_1\Big)\Big).
		\end{align*}
	Moeover, 
	\begin{eqnarray*}
		\bigg[{\mathbb T}^{{\mathcal R}(p^{a},q^{a})}_{m_1},{\mathbb T}^{{\mathcal R}(p^{a},q^{a})}_{m_2}, {\mathbb I}^{{\mathcal R}(p^{a},q^{a})}_{m_3}\bigg]= \frac{1}{\big(\epsilon^{a}_1-\epsilon^{a}_2\big)^{2}}\bigg\{F^{a}_3\,\mathbb{I}^{{\mathcal R}(p^{3\,a},q^{3\,a})}_{m_1+m_2+m_3}-R^{a}_3\,\mathbb{I}^{{\mathcal R}(p^{2\,a},q^{2\,a})}_{m_1+m_2+m_3} \bigg\},
	\end{eqnarray*}
	where
	\begin{eqnarray*}
		F^{a}_3 = \tau^{-a}\epsilon^{a\sum_{s=1}^{2}(z\partial_z-m_s)}_2\big(1-\epsilon^{-a\sum_{s=1}^{2}m_s}_2\big)
	\end{eqnarray*}
	and 
	\begin{eqnarray*}
		R^{a}_3=\epsilon^{a\sum_{s=1}^{2}(z\partial_z-m_s)}_1\big(\epsilon^{-a\sum_{s=1}^{2}m_s}_1 - 1\big)
	\end{eqnarray*}
\end{Remark}
\begin{Remark}
	Interesting cases of Heisenberg Witt $n$-algebras from quantum algebras existing in the literature are deduced as follows:
	\begin{enumerate}
		\item[(a)] Taking $\mathcal{R}(x)=\frac{1-x}{1-q},$ we obtain the $q$-deformed Heisenberg Witt $n$-algebras:
			\begin{align*}
				\Big[{\mathbb T}^{q^{a}}_{m_1},\cdots, {\mathbb T}^{q^{a}}_{m_n}\Big]&=\frac{1}{ \big(1-q^{a}\big)^{n-1}}\Big( M^n_{a}[n]_{q^{a}}{\mathbb T}^{q^{n\,a}}_{m_1+\cdots+m_n}\\&- {[n-1]_{q^{a}}\over  q^{-a\big(\sum_{l=1}^{n}z\partial_z-m_l\big)}}\big(M^n_{a}+ W^n_{a}\big){\mathbb T}^{q^{(n-1)a}}_{m_1+\cdots+m_n}\Big),
			\end{align*}
		where 
			\begin{eqnarray*}
				M^n_{a}=
				\Big(\big(1-q^{a}\big)^{n\choose 2}\prod_{1\leq j < k \leq n}\Big([-m_j]_{q^{a}}-[-m_k]_{q^{a}}\Big)+\prod_{1\leq j < k \leq n}\Big(q^{-a\,m_j}-q^{-a\,m_k}\Big)\Big)
			\end{eqnarray*}
			and 
			\begin{eqnarray*}
				W^{n}_{a}
				=q^{a(n-1)\sum_{s=1}^{n}m_s}\Big(\big(1-q^{a}\big)^{n\choose 2}\prod_{1\leq j < k \leq n}\Big([-m_j]_{q^{a}}-[-m_k]_{q^{a}}\Big)\Big).
			\end{eqnarray*}
		Moreover,
		\begin{eqnarray*}
			\bigg[{\mathbb T}^{q^{a}}_{m_1},{\mathbb T}^{q^{a}}_{m_2},\cdots, {\mathbb I}^{q^{a}}_{m_n}\bigg]= \frac{1}{\big(1-q^{a}\big)^{n-1}}\,F^{a}_n\,\mathbb{I}^{q^{n\,a}}_{m_1+\cdots+m_n},
		\end{eqnarray*}
	with 
	\begin{eqnarray*}
		F^{a}_n = q^{-a}\,q^{a\sum_{s=1}^{n-1}(z\partial_z-m_s)}\big(1-q^{-a\sum_{s=1}^{n-1}m_s}\big).
	\end{eqnarray*}
		For $n=3,$ we deduce the $q$-Heisenberg Witt $3$-algebra:
		\begin{align*}
			\Big[{\mathbb T}^{q^{a}}_{m_1},{\mathbb T}^{q^{a}}_{m_2}, {\mathbb T}^{q^{a}}_{m_3}\Big]&={1\over \big(1-q^{a}\big)^{2}}\Big( M^3_{a}[3]_{q^{a}}{\mathbb T}^{q^{a}}_{m_1+m_2+m_3}\\&- {[2]_{q^{a}}\over  q^{a\big(\sum_{l=1}^{3}z\partial_z-m_l\big)}}\big(M^3_{a}+ W^3_{a}\big){\mathbb T}^{q^{a}}_{m_1+m_2+m_3}\Big),
		\end{align*}
	\begin{eqnarray*}
		\bigg[{\mathbb T}^{q^{a}}_{m_1},{\mathbb T}^{q^{a}}_{m_2}, {\mathbb I}^{q^{a}}_{m_3}\bigg]= \frac{1}{\big(1-q^{a}\big)^{2}}\,F^{a}_3\,\mathbb{I}^{q^{3\,a}}_{m_1+m_2+m_3},
	\end{eqnarray*}
		where 
			\begin{eqnarray*}
				M^3_{a}=
				\Big(\big(1-q^{a}\big)^{3\choose 2}\prod_{1\leq j < k \leq 3}\Big([-m_j]_{q^{a}}-[-m_k]_{q^{a}}\Big)+\prod_{1\leq j < k \leq 3}\big(q^{-a\,m_j}-q^{-a\,m_k}\big)\Big),
			\end{eqnarray*}
			\begin{eqnarray*}
				W^{3}_{a}
				=q^{2\,a\sum_{s=1}^{3}m_s}\Big(\big(1-q^{a}\big)^{n\choose 2}\prod_{1\leq j < k \leq 3}\Big([-m_j]_{q^{a}}-[-m_k]_{q^{a}}\Big),
			\end{eqnarray*}
		and 
		\begin{eqnarray*}
			F^{a}_3 = q^{-a}\,q^{a\sum_{s=1}^{2}(z\partial_z-m_s)}\big(1-q^{-a\sum_{s=1}^{2}m_s}\big).
		\end{eqnarray*}
	\item[(b)] Putting $\mathcal{R}(x,y)=\frac{x-y}{p-q},$ we obtain the $(p,q)$-deformed Heisenberg Witt $n$-algebras:
		\begin{align*}
		\Big[{\mathbb T}^{p^{a},q^{a}}_{m_1},\cdots, {\mathbb T}^{p^{a},q^{a}}_{m_n}\Big]&=\frac{1}{ \big(p^{a}-q^{a}\big)^{n-1}}\Big( M^n_{a}[n]_{p^{a},q^{a}}{\mathbb T}^{p^{n\,a},q^{n\,a}}_{m_1+\cdots+m_n}\\ &- {[n-1]_{p^{a},q^{a}}\over  q^{-a\big(\sum_{l=1}^{n}z\partial_z-m_l\big)}}\big(M^n_{a}+ W^n_{a}\big){\mathbb T}^{p^{(n-1)a},q^{(n-1)\alpha}}_{m_1+\cdots+m_n}\Big),
		\end{align*}
	where 
		\begin{align*}
			M^n_{a}&=
			p^{a(n-1)\sum_{s=1}^{n}m_s}\,\Big(\big(p^{a}-q^{a}\big)^{n\choose 2}\prod_{1\leq j < k \leq n}\Big([-m_j]_{p^{a},q^{a}}-[-m_k]_{p^{a},q^{a}}\Big)\\&+\prod_{1\leq j < k \leq n}\Big(q^{-a\,m_j}-q^{-a\,m_k}\Big)\Big)
		\end{align*}
		and 
		\begin{align*}
			W^{n}_{\alpha}
			&=q^{a(n-1)\sum_{s=1}^{n}m_s}\Big(\big(p^{a}-q^{a}\big)^{n\choose 2}\prod_{1\leq j < k \leq n}\Big([-m_j]_{p^{a},q^{a}}-[-m_k]_{p^{a},q^{a}}\Big)\\&+(-1)^{n-1}\prod_{1\leq j < k \leq n}\Big(p^{-a\,m_j}-p^{-a\,m_k}\Big)\Big).
		\end{align*}
	Moreover, 
	\begin{eqnarray*}
		\bigg[{\mathbb T}^{p^{a},q^{a}}_{m_1},{\mathbb T}^{p^{a},q^{a}}_{m_2},\cdots, {\mathbb I}^{p^{a},q^{a}}_{m_n}\bigg]= \frac{1}{\big(p^{a}-q^{a}\big)^{n-1}}\bigg\{F^{a}_n\,\mathbb{I}^{p^{n\,a},q^{n\,a}}_{m_1+\cdots+m_n}-R^{a}_n\,\mathbb{I}^{p^{(n-1)a},q^{(n-1)a}}_{m_1+\cdots+m_n} \bigg\},
	\end{eqnarray*}
with
\begin{eqnarray*}
	F^{a}_n = (pq)^{-a}\,q^{a\sum_{s=1}^{n-1}(z\partial_z-m_s)}\big(1-q^{-a\sum_{s=1}^{n-1}m_s}\big)
\end{eqnarray*}
and 
\begin{eqnarray*}
	R^{a}_n=p^{a\sum_{s=1}^{n-1}(z\partial_z-m_s)}\big(p^{-a\sum_{s=1}^{n-1}m_s} - 1\big).
\end{eqnarray*}
Taking $n=3,$ we obtain the $(p,q)$-Heisenberg Witt $3$-algebra:
\begin{align}
	\Big[{\mathbb T}^{p^{a},q^{a}}_{m_1},{\mathbb T}^{p^{a},q^{a}}_{m_2}, {\mathbb T}^{p^{a},q^{a}}_{m_3}\Big]&={1\over \big(p^{a}-q^{a}\big)^{2}}\Big( M^3_{a}[3]_{p^{a},q^{a}}{\mathbb T}^{p^{3\,a},q^{3\,a}}_{m_1+m_2+m_3}\\ &- {[2]_{p^{a},q^{a}}\over  q^{a\big(\sum_{l=1}^{3}z\partial_z-m_l\big)}}\big(M^3_{a}+ W^3_{a}\big){\mathbb T}^{p^{2\,a},q^{2\,a}}_{m_1+m_2+m_3}\Big),
\end{align}
\begin{eqnarray*}
	\bigg[{\mathbb T}^{p^{a},q^{a}}_{m_1},{\mathbb T}^{p^{a},q^{a}}_{m_2}, {\mathbb I}^{p^{a},q^{a}}_{m_3}\bigg]= \frac{1}{\big(p^{a}-q^{a}\big)^{2}}\bigg\{F^{a}_n\,\mathbb{I}^{3\,a}_{m_1+m_2+m_3}-R^{a}_3\,\mathbb{I}^{2\,a}_{m_1+m_2+m_3} \bigg\},
\end{eqnarray*}
where 
	\begin{align*}
		M^3_{a}&=
		p^{2\,a\sum_{s=1}^{3}m_s}\Big(\big(p^{a}-q^{a}\big)^{3\choose 2}\prod_{1\leq j < k \leq 3}\Big([-m_j]_{p^{\alpha},q^{a}}-[-m_k]_{p^{a},q^{a}}\Big)\\&+\prod_{1\leq j < k \leq 3}\Big(q^{-a\,m_j}-q^{-a\,m_k}\Big)\Big),
	\end{align*} 
	\begin{align*}
		W^{3}_{a}
		&=q^{2\,a\sum_{s=1}^{3}m_s}\Big(\big(p^{a}-q^{a}\big)^{n\choose 2}\prod_{1\leq j < k \leq 3}\Big([-m_j]_{p^{\alpha},q^{a}}-[-m_k]_{p^{a},q^{a}}\Big)\\&+\prod_{1\leq j < k \leq 3}\Big(p^{-a\,m_j}-p^{-a\,m_k}\Big)\Big),
	\end{align*}
\begin{eqnarray*}
	F^{a}_3 = (pq)^{-a}\,q^{a\sum_{s=1}^{2}(z\partial_z-m_s)}\big(1-q^{-a\sum_{s=1}^{2}m_s}\big)
\end{eqnarray*}
and 
\begin{eqnarray*}
	R^{a}_3=p^{a\sum_{s=1}^{2}(z\partial_z-m_s)}\big(p^{-a\sum_{s=1}^{2}m_s} - 1\big).
\end{eqnarray*}
	\end{enumerate}
\end{Remark}
\subsection{Heisenberg Virasoro constraints and a toy model}
Here, we use the generalized Heisenberg  Virasoro constraints to study a toy model. Particular cases are derived. They  play an important role in the study of matrix models. We consider the generating function with infinitely many parameters
given by \cite{NZ}: $$Z^{toy}(t)=\int \, \,x^{\gamma}\,\exp\left(\displaystyle\sum_{s=0}^{\infty}{t_s\over s!}x^s\right)\,dx,$$
which encodes many different integrals.
We consider the following expansion:
\begin{eqnarray}\label{Bell}
	\exp\left(\displaystyle\sum_{s=0}^{\infty}{t_s\over s!}x^s\right)=\sum_{n=0}^{\infty}B_n(t_1,\cdots,t_n){x^n\over n!},
\end{eqnarray}
where $B_n$ {are the complete} Bell polynomials.

The following property holds for the ${\mathcal R}(p,q)$-deformed derivative
\begin{small}
	\begin{eqnarray*}
		\int_{{\mathbb R}}{\mathcal D}_{{\mathcal R}(p^{a},q^{a})}f(x)d\,x={K(p^{a},q^{a})\over \epsilon^{a}_1-\epsilon^{a}_2}\Big(\int_{-\infty}^{+\infty}{f(\epsilon^{a}_1\,x)\over x}dx -\int_{-\infty}^{+\infty}{f(\epsilon^{a}_2\,x)\over x}dx\Big)=0,
	\end{eqnarray*} 
\end{small}
where 
$$K(p^{a},q^{a})={p^{a}-q^{a}\over p^{P^{a}}-q^{Q^{a}}}{\mathcal R}\big(p^{P^{a}},q^{Q^{a}}\big).$$
For $f(x)=x^{m+\gamma+1}\,\exp\left(\displaystyle\sum_{s=0}^{\infty}{t_s\over s!}x^s\right),$ we have:
\begin{eqnarray*}
	\int_{-\infty}^{+\infty}{\mathcal D}_{{\mathcal R}(p^{a},q^{a})}\left(x^{m+\gamma+1}\,\exp\left(\sum_{s=0}^{\infty}{t_s\over s!}x^s\right)\right)d\,x=0.
\end{eqnarray*}
Thus, 
	\begin{align*}
		{\mathcal D}_{{\mathcal R}(p^{a},q^{a})}\bigg(x^{m+\gamma+1}\,&\exp\left(\displaystyle\sum_{s=0}^{\infty}{t_s\over s!}x^s\right)\bigg)=\frac{[z\partial_z]_{{\mathcal R}(p^{a},q^{a})}\,x^{m+\gamma}}{ \epsilon^{a\,m}_1}\exp\left(\displaystyle\sum_{s=0}^{\infty}{t_s\over s!}x^s\right)\\&+ {K(p^{a},q^{a})\epsilon^{\alpha(m+1+\gamma)}_2\over (\epsilon^{a}_1 - \epsilon^{a}_2)x^{-k-m}}\sum_{k=1}^{\infty}{B_k(t^{a}_1,\cdots,t^{a}_k)\over k!}x^{\gamma}\exp\left(\displaystyle\sum_{s=0}^{\infty}{t_s\over s!}x^s\right),
	\end{align*}
where $t^{a}_k=(\epsilon^{a\,k}_1-\epsilon^{a\,k}_2)t_k.$ Then,  from the constraints on the partition function,
$${\mathbb T}^{\mathcal{R}(p^{a},q^{a})}_m\,Z^{(toy)}(t)=0,\quad m\geq 0$$
and
$${\mathbb I}^{\mathcal{R}(p^{a},q^{a})}_m\,Z^{(toy)}(t)=0,\quad m\geq 0,$$
we have:
\begin{align}\label{hvo1}
{\mathbb T}^{\mathcal{R}(p^{a},q^{a})}_m&=[z\partial_z]_{{\mathcal R}(p^{a},q^{a})}\,m!\, \epsilon^{-a\,m}_1\,{\partial\over \partial t_m}\nonumber\\&+ K(p^{a},q^{a})\frac{\epsilon^{a(m+1+\gamma)}_2}{ \epsilon^{a}_1 - \epsilon^{a}_2}\sum_{k=1}^{\infty}\frac{(k+m)!}{ k!}B_k(t^{a}_1,\cdots,t^{a}_k)\frac{\partial}{ \partial t_{k+m}}.
\end{align}
Similarly, we obtain:
\begin{eqnarray}\label{hvo2}
	{\mathbb I}^{\mathcal{R}(p^{a},q^{a})}_m
	= \tau^{a(m+1+\gamma)}\sum_{k=1}^{\infty}{(k+m)!\over k!}B_k(t^{a}_1,\cdots,t^{a}_k)\frac{\partial}{\partial t_{k+m}}.
\end{eqnarray}
\begin{Remark}
	The Heisenberg Virasoro operators \eqref{hvo1} and \eqref{hvo2} corresponding with quantum algebras in the literature are deduced as:
	\begin{enumerate}
		\item[(i)] The $q$-Heisenberg Virasoro operators:
		\begin{equation*}
		{\mathbb T}^{q^{a}}_m=[z\partial_z]_{q^{a}}\,m!\, q^{-a\,m}\,{\partial\over \partial t_m}+ K(q^{a})\frac{q^{-a(m+1+\gamma)}}{ q^{a} - q^{-a}}\sum_{k=1}^{\infty}\frac{(k+m)!}{ k!}B_k(t^{a}_1,\cdots,t^{a}_k)\frac{\partial}{ \partial t_{k+m}}
		\end{equation*}
and
\begin{equation*}
{\mathbb I}^{q^{a}}_m
= q^{a(m+1+\gamma)}\sum_{k=1}^{\infty}{(k+m)!\over k!}B_k(t^{a}_1,\cdots,t^{a}_k)\frac{\partial}{\partial t_{k+m}},
\end{equation*}
where $$[x]_{q}=\frac{q^x-q^{-x}}{q-q^{-}}.$$
\item[(ii)] The $(p,q)$-Heisenberg Virasoro operartors:
\begin{align*}
{\mathbb T}^{p^{a},q^{a}}_m&=[z\partial_z]_{p^{a},q^{a}}\,m!\, p^{-a\,m}\,{\partial\over \partial t_m}\\&+ K(p^{a},q^{a})\frac{q^{a(m+1+\gamma)}}{ p^{a} - q^{a}}\sum_{k=1}^{\infty}\frac{(k+m)!}{ k!}B_k(t^{a}_1,\cdots,t^{a}_k)\frac{\partial}{ \partial t_{k+m}}
\end{align*}
and
\begin{equation*}
{\mathbb I}^{p^{a},q^{a}}_m
= (pq)^{a(m+1+\gamma)}\sum_{k=1}^{\infty}{(k+m)!\over k!}B_k(t^{a}_1,\cdots,t^{a}_k)\frac{\partial}{\partial t_{k+m}}.
\end{equation*}
	\end{enumerate}
\end{Remark}
\subsection{Generalized  matrix model}

In this section, we generalize the matrix model from the quantum
algebra. Moreover, we {present} the Pochhammer symbol, theta
function, Gaussian density, elliptic gamma function, and the integral
from the $\mathcal{R}(p,q)$-deformed quantum algebra. We {focus} only of the
notions used in the sequel. {More information can be found in \cite{Mironov1}
and references therein.}

We consider now the following relation:
\begin{eqnarray*}
\left \{
\begin{array}{l}
F(z)=z, \\
\\
G(P,Q)=\frac{q^{Q}-p^{P}}{q^{Q}\mathcal{R}(p^P,q^Q)},\quad \mbox{if}\quad \eta>0,
\end{array}
\right .
\end{eqnarray*}
where $\eta$ is given in the relation \eqref{r10}. Then,  
\begin{Definition}
	The  $\mathcal{R}(p,q)$-Pochhammer symbol is given by:
	\begin{eqnarray}\label{rpqPochsyma}
		\big(u,z;\mathcal{R}(p,q)\big)_{n}:=\prod\limits_{j=0}^{n}\left( u-F\big(\frac{q^{j}}{ p^j}\,z\big)G(P,Q)\right),
	\end{eqnarray}
and 
	\begin{eqnarray*}\label{rpqPochsym}
	\big(u,z;\mathcal{R}(p,q)\big)_{\infty}:=\prod\limits_{j=0}^{\infty}\left( u-F\big(\frac{q^{j}}{ p^j}\,z\big)G(P,Q)\right),
	\end{eqnarray*}
	 with 
	the following relation:
	\begin{eqnarray*}
	\big(u,z;\mathcal{R}(p,q)\big)_{n}=\frac{\big(u,z;\mathcal{R}(p,q)\big)_{\infty}}{\big(u,z\frac{q^{n}}{p^n};\mathcal{R}(p,q)\big)_{\infty}}.
	\end{eqnarray*}
 \end{Definition}

Furthermore, the generalized Gaussian density is given as follows:
\begin{eqnarray*}
\rho(z):=\big(u,q^2z^2/\xi^2; \mathcal{R}(p^2,q^2)\big)_{\infty}.
\end{eqnarray*}
\begin{Definition}
  The $\mathcal{R}(p,q)$-deformed matrix model in terms of
  eigenvalue integrals is given by the following relations:
 \begin{align*}\label{rpqpf1}
 Z^{\mathcal{R}(p,q)}_{N}(p_k)&:=\int_{-\xi}^{\xi}\bigg(\prod_{i}\,z^{\beta(N-1)}_{i}\,\rho(z_i)d_{\mathcal{R}(p,q)}z_{i}\bigg)\prod_{j\neq i}\bigg(u,\frac{z_{i}}{z_{j}}; \mathcal{R}(p,q)\bigg)_{\beta}\\&\times \exp\bigg( \sum_{i,k}\frac{p_k}{k}\,z_{i}^k\bigg)
 \end{align*}
 and 
 \begin{align*}
 \bigg\langle \prod_{i}\sum_{m}z_{m}^{k_i} \bigg\rangle_{\mathcal{R}(p,q)}&:=\frac{1}{Z^{\mathcal{R}(p,q)}_{N}(0)}\int_{-\xi}^{\xi}\bigg(\prod_{i}\,z^{\beta(N-1)}_{i}\,\rho(z_i)d_{\mathcal{R}(p,q)}z_{i}\bigg)\\&\times\prod_{j\neq i}\bigg(u,\frac{z_{i}}{z_{j}}; \mathcal{R}(p,q)\bigg)_{\beta}\bigg(\prod_{i}\sum_{m}z_{m}^{k_i}\bigg),
 \end{align*}
 where $\xi$ is a parameter. 
\end{Definition}
\begin{Remark}
	Particular case of matrix models is deduced from the formalism as follows:
	The  $(p,q)$-Pochhammer symbol is given by:
	\begin{eqnarray*}
	\big(u,z;p,q\big)_{\infty}:=\prod\limits_{j=0}^{\infty}\left( u-\frac{q^{j}}{ p^j}\,z\right),
	\end{eqnarray*}
	 with 
	the following relation:
	\begin{eqnarray*}
	\big(u,z;p,q\big)_{n}=\frac{\big(u,z;p,q\big)_{\infty}}{\big(u,z\frac{q^{n}}{p^n};p,q\big)_{\infty}}.
	\end{eqnarray*}
and the $(p,q)$-deformed Gaussian density by:
	\begin{eqnarray}
	\rho(z):=\big(u,q^2z^2/\xi^2; p^2,q^2\big)_{\infty}.
	\end{eqnarray}
	Furthermore, the $(p,q)$-deformed matrix model  is deried by the relations:
		\begin{equation*}
		Z^{p,q}_{N}(p_k):=\int_{-\xi}^{\xi}\bigg(\prod_{i}\,z^{\beta(N-1)}_{i}\,\rho(z_i)d_{p,q}z_{i}\bigg)\prod_{j\neq i}\bigg(u,\frac{z_{i}}{z_{j}}; p,q\bigg)_{\beta} \exp\bigg( \sum_{i,k}\frac{p_k}{k}\,z_{i}^k\bigg)
		\end{equation*}
		and 
		\begin{align*}
		\bigg\langle \prod_{i}\sum_{m}z_{m}^{k_i} \bigg\rangle_{p,q}&:=\frac{1}{Z^{p,q}_{N}(0)}\int_{-\xi}^{\xi}\bigg(\prod_{i}\,z^{\beta(N-1)}_{i}\,\rho(z_i)d_{p,q}z_{i}\bigg)\\&\times\prod_{j\neq i}\bigg(u,\frac{z_{i}}{z_{j}}; p,q\bigg)_{\beta}\bigg(\prod_{i}\sum_{m}z_{m}^{k_i}\bigg),
		\end{align*}
		where $\xi$ is a parameter.
\end{Remark}
Now, we investigate the
elliptic generalized matrix models. 
\begin{Definition}
	The elliptic $\mathcal{R}(p,q)$-Pochhammer symbol is defined as follows:
	\begin{eqnarray}
	\big(u,z;\mathcal{R}(p,q),w\big)_{\infty}:=\prod\limits_{j,k=0}^{\infty}\left( u-\gamma_{j,k}(z,w)\right),
	\end{eqnarray}
	$\gamma_{j,k}(z,w)=F\big(\frac{q^{j}}{ p^j}\,w^k\,z\big)G(P,Q).$
	
	 Moreover, 
	the $\mathcal{R}(p,q)$-theta function $\Theta(u,z; \mathcal{R}(p,q))$ is given by:
	\begin{eqnarray}\label{theta2}
	\theta_{w}(u,z)=\big(u,z;w\big)_{\infty}\big(u,w/z;w\big)_{\infty}.
	\end{eqnarray}
Furthermore, the generalized elliptic gamma function is defined by:
	\begin{eqnarray*}\label{Rpqegf}
	\Gamma\big(u,z;w,\mathcal{R}(p,q)\big):=\frac{\big(u,qw/z;w,\mathcal{R}(p,q)\big)_{\infty}}{\big(u,z;w,\mathcal{R}(p,q)\big)_{\infty}}.
	\end{eqnarray*}
\end{Definition}
In the particular case, we have:
\begin{eqnarray*}
\Gamma\big(u,q^n;w,\mathcal{R}(p,q)\big)=\prod_{k=1}^{\infty}\frac{[k]_{\mathcal{R}(s,w)}}{[k]_{\mathcal{R}(p,q)}}\,\prod_{i=1
}^{n-1}\theta_{w}(u,q^{i}).
\end{eqnarray*}
We consider the relation
\begin{eqnarray}\label{id}
\langle f(z)\rangle:= \frac{\int_{-\xi}^{\xi}\rho(z)\,f(z)d_{\mathcal{R}(p,q)}z}{\int_{-\xi}^{\xi}\rho(z)\,d_{\mathcal{R}(p,q)}z}.
\end{eqnarray}
Then, from the {generalized} Andrews-Askey formula \cite{GR}:
	\begin{align*}
	\int_{-\xi}^{\xi}\frac{\big(u,q^2z^2/\xi^2; \mathcal{R}(p^2,q^2)\big)_{\infty}}{
		\big(u,-\alpha_1z/\xi; \mathcal{R}(p,q)\big)_{\infty}\big(u,\alpha_2z/\xi; \mathcal{R}(p,q)\big)_{\infty}}&d_{\mathcal{R}(p,q)}z=\xi(p-q)\frac{(u,q^2; \mathcal{R}(p^2,q^2))_{\infty}}{(u,\alpha^2_1; \mathcal{R}(p^2,q^2))_{\infty}}\\&\times\frac{(u,-1; \mathcal{R}(p,q))_{\infty}(u,\alpha_1\alpha_2; \mathcal{R}(p,q))_{\infty}}{(u,\alpha^2_2; \mathcal{R}(p^2,q^2))_{\infty}}.
	\end{align*}
	For $\alpha_1=\alpha_2=\alpha,$ the above relation takes the following form:
	\begin{align*}
	\int_{-\xi}^{\xi}\frac{\big(u,q^2z^2/\xi^2; \mathcal{R}(p^2,q^2)\big)_{\infty}}{
		\big(u,\alpha^2z^2/\xi^2; \mathcal{R}(p^2,q^2)\big)_{\infty}}d_{\mathcal{R}(p,q)}z	&=\xi(p-q)\frac{\big(u,q^2; \mathcal{R}(p^2,q^2)\big)_{\infty}\big(u,-1; \mathcal{R}(p,q)\big)_{\infty}}{\big(u,\alpha^2; \mathcal{R}(p^2,q^2)\big)_{\infty}\big)}\\&\times\frac{\big(u,\alpha^2; \mathcal{R}(p,q)\big)_{\infty}}{\big((u,\alpha^2; \mathcal{R}(p^2,q^2))_{\infty}\big)}
	\end{align*}
	and can be rewritten as:
	\begin{align*}
	\int_{-\xi}^{\xi}\frac{\big(u,q^2z^2/\xi^2; \mathcal{R}(p^2,q^2)\big)_{\infty}}{
		\big(u,\alpha^2z^2/\xi^2; \mathcal{R}(p^2,q^2)\big)_{\infty}}&d_{\mathcal{R}(p,q)}z=\xi(p-q)\prod_{n=0}^{\infty}\frac{\big(u-F\big(\frac{q^{2n+2}}{ p^{2n+2}}\big)G(P,Q)\big)}{\big(u-F\big(\frac{q^{2n}}{ p^{2n}}\big)G(P,Q)\alpha^2\big)}\\&\times\prod_{n=0}^{\infty}\frac{\big(u+F\big(\frac{q^{n}}{ p^{n}}\big)G(P,Q)\big)\big(u-F\big(\frac{q^{n}}{ p^{n}}\big)G(P,Q)\alpha^2\big)}{\big(u-F\big(\frac{q^{2n}}{ p^{2n}}\big)G(P,Q)\alpha^2\big)} .
	\end{align*}
	Taking $\alpha=0,$ we have
	\begin{eqnarray*}
	\int_{-\xi}^{\xi}\,\rho(z)d_{\mathcal{R}(p,q)}z=\xi(p-q)\prod_{n=0}^{\infty}{\big(u-F\big(\frac{q^{2n+2}}{ p^{2n+2}}\big)G(P,Q)\big)\big(u+F\big(\frac{q^{n}}{ p^{n}}\big)G(P,Q)\big)}.
	\end{eqnarray*}
	Then, from the relation \eqref{id}, we obtain 
	\begin{equation}\label{ab}
	\bigg
	\langle \frac{1}{\big(u,q^2z^2/\xi^2; \mathcal{R}(p^2,q^2)\big)_{\infty}}\bigg\rangle=\prod_{n=0}^{\infty}\frac{\big(u-F\big(\frac{q^{n}}{ p^{n}}\big)G(P,Q)\alpha^2\big)
	}{\big(u-F\big(\frac{q^{2n}}{ p^{2n}}\big)G(P,Q)\alpha^2\big)^2}.
	\end{equation}
	Using the relations
	\begin{equation*}
	\frac{1}{\big(u,q^2z^2/\xi^2; \mathcal{R}(p^2,q^2)\big)_{\infty}}=\sum_{i=0}^{\infty}\,\frac{1}{\big(u,q^2; \mathcal{R}(p^2,q^2)\big)_{i}}\bigg(\frac{\alpha\,z}{\xi}\bigg)^{2i}
	\end{equation*}
	and 
	\begin{equation*}
	\big(u,z;\mathcal{R}(p,q)\big)_{\infty}=\exp\bigg( -\sum_{i}\frac{z^i}{i[i]_{\mathcal{R}(p,q)}}\bigg),
	\end{equation*} 
the relation \eqref{ab} is reduced as:
\begin{eqnarray*}
\bigg
\langle \frac{1}{\big(u,q^2z^2/\xi^2; \mathcal{R}(p^2,q^2)\big)_{\infty}}\bigg\rangle&=&\sum_{i=0}^{\infty}\,\frac{1}{\big(u,q^2; \mathcal{R}(p^2,q^2)\big)_{i}}\bigg(\frac{\alpha}{\xi}\bigg)^{2i}\langle z^{2i} \rangle \nonumber\\&=&\exp\bigg\{\sum_{i}\frac{\alpha^{2i}}{i}\bigg( \frac{2}{[2i]_{\mathcal{R}(p,q)}}-\frac{1}{[i]_{\mathcal{R}(p,q)}}\bigg)\bigg\}.
\end{eqnarray*}
Then, 
	the following relation holds:
	\begin{eqnarray}\label{measure}
	\langle z^k \rangle=\frac{1}{2}\xi^{k}.\delta^{(2)}_k.\prod_{i=1}^{k/2}\big(u-F\big(\frac{q^{2i-1}}{p^{2i-1}}\big)G(P,Q)\big).
	\end{eqnarray}

        Note that,  to define the generalized elliptic matrix model, we need to define the elliptic generalization of the Vandermonde factor and  measure from
        the {relation} \eqref{measure}. 
Then, the elliptic analogues of the relation \eqref{measure} can be deduced as follows:
\begin{eqnarray*}
\langle z^k \rangle_{(\rm ell)}=\xi^{k}.\delta^{(2)}_k.\prod_{i=1}^{k/2} \theta_{w}(u,q^{2i-1})
\end{eqnarray*}
and the elliptic Vandermonde factor is provided by the elliptic gamma function. Moreover, the elliptic Gaussian density is given by $$\rho^{(\rm ell)}(z,w)=\big(u,q^2z^2/\xi^2;w, \mathcal{R}(p^2,q^2)\big)_{\infty}.$$ Then, the  definition follows:
\begin{Definition}
	The generalized elliptic matrix models is defined as:
	\begin{align*}
	Z_N^{\rm ell}(\{ p_k\})&=\int \bigg(\prod_{i}z^{\beta(N-1)}_{i
	}\rho^{(\rm ell)}(z_{i})d_{\rm ell}z_{i}\bigg)\\&\times\prod_{j\neq i}\frac{\Gamma\big(u,q^{\beta},\frac{z_{i}}{z_{j}};w,\mathcal{R}(p,q)\big)}{\Gamma\big(u,\frac{z_{i}}{z_{j}};w,\mathcal{R}(p,q)\big)}\,\exp\bigg(\sum_{i,k}\frac{p_k}{k}z^k_{i}\bigg)
	\end{align*}
	and 
	\begin{align*}
	\bigg
	\langle \prod_{i}\sum_{m}z^{k_i}_{m}\bigg\rangle_{(\rm ell)}&=\frac{1}{Z^{(ell)}_{N}(0)}\int \bigg(\prod_{i}z^{\beta(N-1)}_{i
	}\rho^{(\rm ell)}(z_{i})d_{\rm ell}z_{i}\bigg)\\&\times\prod_{j\neq i}\frac{\Gamma\big(u,q^{\beta},\frac{z_{i}}{z_{j}};w,\mathcal{R}(p,q)\big)}{\Gamma\big(u,\frac{z_{i}}{z_{j}};w,\mathcal{R}(p,q)\big)}\,\bigg(\prod_{i}\sum_{m}z^{k_i}_{m}\bigg).
	\end{align*}
\end{Definition} 
\begin{Remark}
	Particular case of elliptic matrix models  is recovered as follows:
	The elliptic $(p,q)$-Pochhammer symbol is defined as follows:
	\begin{eqnarray*}
	\big(u,z,w;p,q\big)_{\infty}:=\prod\limits_{j,k=0}^{\infty}\left( u-\gamma_{j,k}(z,w)\right),
	\end{eqnarray*}
	$\gamma_{j,k}(z,w)=F\big(\frac{q^{j}}{ p^j}\,w^k\,z\big)G(P,Q).$ Moreover, 
	the $(p,q)$-theta function $\Theta(u,z; p,q)$ is given by:
	\begin{eqnarray*}
	\theta_{w}(u,z)=\big(u,z;w\big)_{\infty}\big(u,w/z;w\big)_{\infty}
	\end{eqnarray*}
	and the $(p,q)$-deformed elliptic gamma function as:
	\begin{eqnarray*}
	\Gamma\big(u,z;w,p,q\big):=\frac{\big(u,qw/z;w,p,q\big)_{\infty}}{\big(u,z;w,p,q\big)_{\infty}}.
	\end{eqnarray*}
Moreover, the $(p,q)$-elliptic Gaussian density is given by $$\rho^{(\rm ell)}(z,w)=\big(u,q^2z^2/\xi^2;w, p^2,q^2\big)_{\infty}.$$ and 
	the $(p,q)$-elliptic matrix models by:
		\begin{align*}
		Z_N^{\rm ell}(\{ p_k\})&=\int \bigg(\prod_{i}z^{\beta(N-1)}_{i
		}\rho^{(\rm ell)}(z_{i})d_{\rm ell}z_{i}\bigg)\\&\times\prod_{j\neq i}\frac{\Gamma\big(u,q^{\beta},\frac{z_{i}}{z_{j}};w,p,q\big)}{\Gamma\big(u,\frac{z_{i}}{z_{j}};w,p,q\big)}\exp\bigg(\sum_{i,k}\frac{p_k}{k}z^k_{i}\bigg)
		\end{align*}
		and 
		\begin{align*}
		\bigg
		\langle \prod_{i}\sum_{m}z^{k_i}_{m}\bigg\rangle_{(\rm ell)}&=\frac{1}{Z^{(ell)}_{N}(0)}\int \bigg(\prod_{i}z^{\beta(N-1)}_{i
		}\rho^{(\rm ell)}(z_{i})d_{\rm ell}z_{i}\bigg)\\&\times\prod_{j\neq i}\frac{\Gamma\big(u,q^{\beta},\frac{z_{i}}{z_{j}};w,p,q\big)}{\Gamma\big(u,\frac{z_{i}}{z_{j}};w,p,q\big)}\,\bigg(\prod_{i}\sum_{m}z^{k_i}_{m}\bigg).
		\end{align*}
\end{Remark}

\begin{Definition}
	The $\mathcal{R}(p,q)$-differential oper{a}tor is defined as follows:
	\begin{eqnarray}
	T^{\mathcal{R}(p,q)}_n\phi(z) := -\sum\limits_{l=1}^N {\mathcal D}_{\mathcal{R}(p,q)}^{z_l}\,z_l^{n+1}\phi(z),\label{rpqVir-matr-op}
	\end{eqnarray}
	which acts on the functions of $N$ variables  and $ {\mathcal D}_{\mathcal{R}(p,q)}^{z_l}$ is $\mathcal{R}(p,q)$-derivative with respect to the $z_l$-variable.
\end{Definition}
\begin{Proposition}
	The operators \eqref{rpqVir-matr-op} verify the $\mathcal{R}(p,q)$-deformed commutation relation:
	\begin{eqnarray*}
	\big[T^{\mathcal{R}(p,q)}_n,T^{\mathcal{R}(p,q)}_m\big]_{x_n, x_m}= \big([n]_{\mathcal{R}(p,q)}-[m]_{\mathcal{R}(p,q)}\big)T^{\mathcal{R}(p,q)}_{n+m},
	\end{eqnarray*}
	where 
	\begin{eqnarray*}
	x_n=q^{n-m}p^n\,\chi_{nm}(p,q)\mbox{,}\quad x_m=p^n\chi_{nm}(p,q)
	\end{eqnarray*}
	and
	\begin{eqnarray*}
	\chi_{nm}(p,q)=\frac{[n]_{\mathcal{R}(p,q)}-[m]_{\mathcal{R}(p,q)}}{[n+1]_{\mathcal{R}(p,q)}-(pq)^{n-m}[m+1]_{\mathcal{R}(p,q)}}.
	\end{eqnarray*}
\end{Proposition}
We can rewrite the above relation by: 
\begin{eqnarray*}
\big[T^{\mathcal{R}(p,q)}_n,T^{\mathcal{R}(p,q)}_m\big]_{x_{n+1}, x_{m+1}}= \big([n+1]_{\mathcal{R}(p,q)}-[m+1]_{\mathcal{R}(p,q)}\big)T^{\mathcal{R}(p,q)}_{n+m}.
\end{eqnarray*}
\begin{Proposition}
	The $\mathcal{R}(p,q)$-operator \eqref{rpqVir-matr-op} can be given as follows:
		\begin{align}
		T_n^{\mathcal{R}(p,q)}&=\frac{K(P,Q)}{p-q}\bigg[(\frac{q}{p})^{n+1+\beta(N-1)} \sum\limits_{l=0}^\infty \frac{(l+n-2N)!}{l!} B_l(\tilde{t}_1, ... , \tilde{t}_l)\nonumber\\&\times
		{\it D}_N \frac{\partial}{\partial t_{l+n-2N}} - p^{n+1+\beta(N-1)} n!\frac{\partial}{\partial t_n}\bigg],
		\label{rpqVir:operator}
		\end{align}
		where ${\it D}_{N}$ is a differential operator\eqref{diffop}.
\end{Proposition}
\begin{proof}
	The elliptic generalized matrix model can be rewritten as:
	\begin{align}\label{integral}
	Z_N^{\rm ell}(\{ p_k\})&=\int \prod_{i}d_{\rm ell}z_{i}\,\prod_{i}z^{\beta(N-1)}_{i
	}\rho^{(\rm ell)}(z_{i})\nonumber\\&\times\prod_{j\neq i}\frac{\Gamma\big(u,q^{\beta},\frac{z_{i}}{z_{j}};w,\mathcal{R}(p,q)\big)}{\Gamma\big(u,\frac{z_{i}}{z_{j}};w,\mathcal{R}(p,q)\big)}\,\exp\bigg(\sum_{i,k}\frac{p_k}{k}z^k_{i}\bigg).
	\end{align}
Putting the $\mathcal{R}(p,q)$-differential  operators \eqref{rpqVir-matr-op} under the  integral (\ref{integral}), we obtain naturally zero. 
Now we have to evaluate how these differential operators act on the integrand. 
 Setting
\begin{eqnarray*}
g(z)=\prod_{i}z^{\beta(N-1)}_{i
}\rho^{(\rm ell)}(z_{i})\prod_{j\neq i}\frac{\Gamma\big(u,q^{\beta},\frac{z_{i}}{z_{j}};w,\mathcal{R}(p,q)\big)}{\Gamma\big(u,\frac{z_{i}}{z_{j}};w,\mathcal{R}(p,q)\big)}
\end{eqnarray*}
and $$f(z)=z^{n+1},$$
we have:  
	\begin{eqnarray}\label{useful-id}
	T^{\mathcal{R}(p,q)}_ng(z)
	&=& 
	\sum_{l=1}^N\frac{K(P,Q)}{p-q}\bigg(\big(\frac{q}{p}\big)^{n+1+\beta(N-1)}\prod_{j\neq l}\frac{p}{q}\frac{z_j^2}{z_l^2}-1\bigg)
	p^{n+1+\beta(N-1)}z^{n}_l\nonumber\\&\times&\prod_{i}(zp)^{\beta(N-1)}_{i
	}\rho^{(\rm ell)}(pz_{i})\prod_{j\neq i}\frac{\Gamma\big(u,q^{\beta},\frac{z_{i}}{z_{j}};w,\mathcal{R}(p,q)\big)}{\Gamma\big(u,\frac{z_{i}}{z_{j}};w,\mathcal{R}(p,q)\big)},
	\end{eqnarray}
where 
\begin{eqnarray*}
K(P,Q)=\frac{p-q}{p^P-q^Q}\,\mathcal{R}(p^P,q^Q).
\end{eqnarray*}

%Let us now recall some notion about the Bell polynomials. The exponent can be expanded as follows
%\begin{eqnarray}
%e^{\sum\limits_{s=1}^\infty \frac{t_s}{s!} x^s} = \sum\limits_{n=0}^\infty B_n (t_1, t_2, ... , t_n) \frac{x^n}{n!}~,\label{Bell-exponent}
%\end{eqnarray}
%where    is 
 The $n$th complete Bell polynomial $B_n$ given by \eqref{Bell} satisfy the following relations:
%For  the calculation of the $\mathcal{R}(p,q)$-derivative of the exponental factor  we use the expansion (\ref{Bell-exponent})  in terms of the  Bell 
%polynomials $B_k$
\begin{eqnarray}\label{Bell-pr3}
B_l (\tilde{t}_1, ... ,\tilde{t}_l) = \sum\limits_{\nu=0}^l q^{\nu} \binom{l}{\nu} B_{\nu} (t_1, ... , t_{\nu}) B_{n-\nu} (-t_1, ... , - t_{n-\nu}),
\end{eqnarray}
where $\tilde{t}_k = (q^k-1) t_k,$ 
and 
	\begin{align}
	\exp\left(\sum\limits_{k=1}^{\infty}\frac{t_k}{k!}q^k z_i^k\right)&=
	\sum_{k=0}^{\infty}\sum_{\nu=0}^{\infty}\frac{1}{k!\nu!}B_k(t_1,\dots, t_k)\nonumber\\&\times B_{\nu}(-t_1,\dots, -t_{\nu})q^k z_i^{k+\nu}
	\exp\bigg(\sum\limits_{l=1}^{\infty}\frac{t_l}{l!} z_i^l\bigg)\nonumber\\
	&=\sum_{k=0}^{\infty}\frac{1}{k!}B_k\left(\tilde{t}_1,\dots,\tilde{t}_k\right)x^k\exp\left(\sum\limits_{l=1}^{\infty}\frac{t_l}{l!} z_i^l\right)
	.\label{Bells-useful}
	\end{align}
Applying the formulas (\ref{useful-id}) and (\ref{Bells-useful}),  we find  the insertion of the $\mathcal{R}(p,q)$-operator (\ref{rpqVir-matr-op}) under   the integral (\ref{integral}).

Then, the relation \eqref{useful-id} can be rewritten in the simpler form:
	\begin{align*}
	T^{\mathcal{R}(p,q)}_ng(z)&=\frac{K(P,Q)}{p-q}\bigg[\prod_{j=1}^{N} z_j^2 \sum_{l=1}^{N}\sum_{k,\nu=0}^{\infty}\big(\frac{q}{p}\big)^{n+1+\beta(N-1)}q^{k}\frac{1}{k!\nu!}B_k(t_1,\dots,t_k)
	\\ &\times B_{\nu}(-t_1,\dots,-t_{\nu})z_l^{k+\nu+n-2N}- p^{n+1+\beta(N-1)}\,\sum_{l=1}^{N} z_l^n\bigg]~.\label{gen-terms-elvir}
	\end{align*}

Using the Newton's identities, 
\begin{eqnarray*}
\prod_{i=1}^N z_i=\frac{1}{N!}\left|\begin{array}{cccccc}
\nu_1 & 1 & 0 & \dots &  &  \\
\nu_2  & \nu_1 & 2 & 0 & \dots & \\
\dots   &  \dots  &\dots &  \dots  &\dots &\\
\nu_{N-1} & \nu_{N-2} & \dots  &\dots& \nu_1 & N-1\\
\nu_{N} & \nu_{N-1} & \dots  &\dots& \nu_2 & \nu_1 
\end{array}\right|~,
\end{eqnarray*}
where $\nu_k\equiv \sum\limits_{i=1}^{N}z_{i}^k ,$   the terms $\sum\limits_{i=1}^{N}z_{i}^k$ may be generated by taking 
the derivatives with respect to $t$  and thus we can consider the following differential operator
\begin{eqnarray}\label{diffop}
{\it D}_N =\frac{1}{N!}\left|\begin{array}{cccccc}
2! \frac{\partial}{\partial t_2} & 1 & 0 & \dots &  &  \\
4!\frac{\partial}{\partial t_4}  &2! \frac{\partial}{\partial t_2} & 2 & 0 & \dots & \\
\dots   &  \dots  &\dots &  \dots  &\dots &\\
(2N-2)! \frac{\partial}{\partial t_{2N-2}} & (2N-4)!\frac{\partial}{\partial t_{2N-4}} & \dots  &\dots& 2!\frac{\partial}{\partial t_2} & N-1\\
(2N)! \frac{\partial}{\partial t_{2N}} & (2N-2)!\frac{\partial }{\partial t_{2N-2}} & \dots  &\dots& 4!\frac{\partial}{\partial t_4} & 2! \frac{\partial}{\partial t_2} 
\end{array}\right|~,
\end{eqnarray}
with the property that
\begin{eqnarray*}
\prod_{j=1}^{N} z_j^2 ~ e^{\sum\limits_{k=0}^{\infty} \frac{t_k}{k!}\sum\limits_{i=1}^N z_i^k  }=
{\it D}_N  \left (e^{\sum\limits_{k=0}^{\infty} \frac{t_k}{k!}\sum\limits_{i=1}^N z_i^k  } \right )~.
\end{eqnarray*}
Combining all together we obtain the following $\mathcal{R}(p,q)$-Virasoro operator:
	\begin{eqnarray*}
	T^{\mathcal{R}(p,q)}_n &=& \frac{K(P,Q)}{p-q}\bigg[  \sum_{k,\nu=0}^{\infty}\big(\frac{q}{p}\big)^{n+1-\beta(N-1)}q^{k}\frac{(k+\nu+n-2N)!}{k!\nu!}B_k(t_1,\dots,t_k)\nonumber\\&\times&
	B_{\nu}(-t_1,\dots,-t_{\nu}) {\it D}_N \frac{\partial}{\partial t_{k+\nu+n-2N}} - 
	p^{n+1+\beta(N-1)}\,n!\frac{\partial}{\partial t_n}\bigg]~,\label{rpqVir-exp-matrix}
	\end{eqnarray*}
which annihilates the generating function $Z_N^{\rm ell}(\{ t\})$. Using the property (\ref{Bell-pr3}), the result follows.  
\end{proof}
%We see that generically the operators $T_n^q$ are higher order differential operators and action of these operators
%on $Z_N^{\rm ell}(\{ t\})$ generates the insertion of the terms (\ref{gen-terms-elvir}) under the integral. 
%It is crucial to stress that there are many different higher order operators which will generate exactly the same 
%insertion  (\ref{gen-terms-elvir}).

Now, we can show that the $\mathcal{R}(p,q)$-operators (\ref{rpqVir:operator}) obey  the following commutation relation:
\begin{small}
	\begin{eqnarray*}
	[T_n^{\mathcal{R}(p,q)}, T_m^{\mathcal{R}(p,q)}] = f_{nm}(p,q)([n]_{{\mathcal{R}(p,q)}} - [m]_{{\mathcal{R}(p,q)}}) \bigg( [2]_{{\mathcal{R}(p,q)}} T^{\mathcal{R}(p^2,q^2)}_{n+m} - T^{{\mathcal{R}(p,q)}}_{n+m} \bigg),
	\label{rpqVir:new}    
	\end{eqnarray*}
\end{small} 
where $f_{nm}(p,q)$ is the function depending on $p,q,n,$ and $m$ and  $T^{\mathcal{R}(p^2,q^2)}_n$ is the
$\mathcal{R}(p,q)$- difference operator defined by:
\begin{eqnarray*}
T^{\mathcal{R}(p^2,q^2)}_n\phi(z) = -\sum\limits_{l=1}^N {\mathcal D}_{\mathcal{R}(p^2,q^2)}^{z_l}\,z_l^{n+1}\phi(z).
\end{eqnarray*} 

From the above procedure, we can deduce that the operators $T^{\mathcal{R}(p^2,q^2)}_n$ also  annihilate the $\mathcal{R}(p,q)$-generating function
$Z_N^{\rm ell}(\{ t\}).$  Then, we have: 
 $$T_{n}^{\mathcal{R}(p^2,q^2)} Z_N^{\rm ell}(\{ t\})=0,$$ where the $\mathcal{R}(p,q)$-differential operator $T_{n}^{\mathcal{R}(p^2,q^2)}$ is given by the following relation using \eqref{rpqVir:operator}:
\begin{eqnarray}\label{vir:squared}
T_{n}^{\mathcal{R}(p^2,q^2)}&=&\frac{K(P,Q)}{p^2-q^2}\bigg[ \big(\frac{q}{p}\big)\big)^{2n+4-4\beta(N-1)} \sum\limits_{l=0}^\infty \frac{(l+n-4N)!}{l!} B_l(\hat{t}_1, ... , \hat{t}_l)\nonumber\\&\times&
{\tilde{\it D}}_N \frac{\partial}{\partial t_{l+n-4N}} - p^{2n+4+4\beta(N-1)} n!\frac{\partial}{\partial t_n}\bigg],
\end{eqnarray}
with ${\tilde{\it D}}_N$ the differential operator defined by: 
\begin{eqnarray*}
{ \tilde{\it D}}_N =\frac{1}{N!}\left|\begin{array}{cccccc}
4! \frac{\partial}{\partial t_4} & 1 & 0 & \dots &  &  \\
8!\frac{\partial}{\partial t_8}  &4! \frac{\partial}{\partial t_4} & 2 & 0 & \dots & \\
\dots   &  \dots  &\dots &  \dots  &\dots &\\
(4N-4)! \frac{\partial}{\partial t_{4N-4}} & (4N-8)!\frac{\partial}{\partial t_{4N-8}} & \dots  &\dots& 4!\frac{\partial}{\partial t_4} & N-1\\
(4N)! \frac{\partial}{\partial t_{4N}} & (4N-4)!\frac{\partial }{\partial t_{4N-4}} & \dots  &\dots& 8!\frac{\partial}{\partial t_8} & 4! \frac{\partial}{\partial t_4} 
\end{array}\right|.
\end{eqnarray*}
From the relation \eqref{vir:squared}, we see that the operators $T_{n}^{\mathcal{R}(p^2,q^2)}$ are higher order differential operators. Similarly, using the same procedure, we can define 
the operators $T^{\mathcal{R}(p^j,q^j)}_n$ as follows: 
\begin{eqnarray*}
	T_{n}^{\mathcal{R}(p^j,q^j)}&=&\frac{K(P,Q)}{p^j-q^j}\bigg[ q^{jn+j^2-j^2\beta(N-1)} \sum\limits_{l=0}^\infty \frac{(l+n-2jN)!}{l!} B_l(\hat{t}_1, ... , \hat{t}_l)\nonumber\\&\times&
	\mathcal{\hat{D}}_N \frac{\partial}{\partial t_{l+n-2jN}} - p^{jn+j^2\beta(N-1)} n!\frac{\partial}{\partial t_n}\bigg],
\end{eqnarray*}
with $\mathcal{\hat{D}}_N$  given by: 
\begin{eqnarray*}
	{\it \hat{D}}_N =\frac{1}{N!}\left|\begin{array}{cccccc}
		2j! \frac{\partial}{\partial t_{2j}} & 1 & 0 & \dots &  &  \\
		4j!\frac{\partial}{\partial t_{4j}}  &2j! \frac{\partial}{\partial t_{2j}} & 2 & 0 & \dots & \\
		\dots   &  \dots  &\dots &  \dots  &\dots &\\
		(2N-2)j! \frac{\partial}{\partial t_{(2N-2)j}} & (2N-4)j!\frac{\partial}{\partial t_{(2N-4)j}} & \dots  &\dots& 2j!\frac{\partial}{\partial t_{2j}} & N-1\\
		(2jN)! \frac{\partial}{\partial t_{2jN}} & (2N-2)j!\frac{\partial }{\partial t_{(2N-2)j}} & \dots  &\dots& 4j!\frac{\partial}{\partial t_{4j}} & 2j! \frac{\partial}{\partial t_{2j}} 
	\end{array}\right|.
\end{eqnarray*}
\section{Concluding remarks}
We have constructed the $\mathcal{R}(p,q)$-deformed Heisenberg-Virasoro algebra,  the $\mathcal{R}(p,q)$-Heisenberg-Witt $n$-algebra. Moreover, we have generalized the matrix models,  the elliptic hermitian matrix models and  presented  the $\mathcal{R}(p,q)$-differential operatos of the Virasoro algebra. Related particular cases have been deduced.
\section*{Acknowledgements}
This work is supported by {a} DAAD research stay, {reference
91819215. The research of RW is
funded by the Deutsche Forschungsgemeinschaft (DFG, German
Research Foundation) -- Project-ID 427320536 -- SFB 1442, as well as
under Germany's Excellence Strategy EXC 2044 390685587,
Mathematics M\"unster: Dynamics--Geometry--Structure.}

\end{document}